\documentclass[10pt]{amsart}
\usepackage{amsmath, amscd, amsthm} 
\usepackage{amssymb, amsfonts} 
\usepackage{enumerate}
\usepackage{bbm}
\usepackage{color}
\usepackage[all]{xy} %xy-pic for diagrams.
\usepackage[left=1in, right=1.5in]{geometry}
\usepackage{palatino}
\usepackage{mathrsfs}
\usepackage{graphicx}
\usepackage{booktabs} 
\usepackage{aliascnt}
\usepackage{mathrsfs}
\usepackage[dvipsnames]{xcolor} 
\usepackage{hyperref}
\definecolor{dark-red}{rgb}{0.4,0.15,0.15}
\hypersetup{colorlinks, linkcolor= Red, citecolor= blue, urlcolor= Green}
\usepackage{mathrsfs}
\usepackage{tikz-cd}
\usepackage{stackrel}
\usepackage[capitalise]{cleveref}
\usepackage[toc,page]{appendix}
\usepackage{changepage}
\usepackage{quiver}

\usepackage[backend=biber, backref=true, style=alphabetic, sorting=nty, maxbibnames=99, maxalphanames=999]{biblatex}
\DefineBibliographyStrings{english}{%
  backrefpage = {\textbf{P.}}, % originally "cited on page"
  backrefpages = {\textbf{P.}}, % originally "cited on pages"
}

\renewcommand{\AA}{\mathbb{A}}

\newcommand{\CC}{\mathbb{C}}
\newcommand{\GG}{\mathbb{G}}

\newcommand{\HH}{\mathcal{H}}
\newcommand{\KK}{\mathcal{K}}
\newcommand{\SH}{\mathrm{SH}}

\newcommand{\ZZ}{\mathbb{Z}}

\newcommand{\Sm}{\mathrm{Sm}}

\newcommand{\pr}{\mathrm{pr}}

\newcommand{\Spec}{\operatorname{Spec}}

\newcommand{\ML}{\operatorname{ML}}

\newcommand{\bs}{\backslash}
\newcommand{\Abs}{\mathbb{A}^2\backslash\{0\}}
\newcommand{\Absn}{\AA^n\backslash\{0\}}
\renewcommand{\frak}{\mathfrak}

\newcommand{\RR}{\mathbb{R}}
\renewcommand{\L}{\mathbb{L}}

\newcommand{\FF}{\mathbb{F}}

\renewcommand{\SS}{\mathbb{S}}

\newcommand{\Th}{\mathrm{Th}}
\newcommand{\id}{\mathrm{id}}

\newcommand{\mcal}[1]{\mathcal{#1}}
\newcommand{\ul}[1]{\underline{\smash{#1}}}

\newcommand{\barg}{\bar{g}}

\newcommand{\Lin}{\langle}
\newcommand{\Rin}{\rangle}

\newcommand{\XX}{\mathcal{X}}
\newcommand{\QQ}{\mathbb{Q}} 
\newcommand{\PP}{\mathbb{P}}

\newtheorem*{theorem*}{Theorem}
\newtheorem*{prop*}{Proposition}
\newtheorem*{qstn*}{Question}
\newtheorem*{lemma*}{Lemma}
\newtheorem*{corollary*}{Corollary}
\newtheorem*{defn*}{Definition}

\DeclareMathOperator{\Hom}{Hom}

\DeclareMathOperator{\Nis}{Nis}

\DeclareMathOperator{\SL}{SL}

\usepackage[T1]{fontenc}

\numberwithin{equation}{section} % Equations track separately (e.g., 5.1, 5.2)

\theoremstyle{definition}

\newtheorem{theorem}{Theorem}[section]

\newaliascnt{prop}{theorem}  
\newtheorem{prop}[prop]{Proposition}
\aliascntresetthe{prop}

\newaliascnt{lemma}{theorem}  % <-- Changed from equation to theorem
\newtheorem{lemma}[lemma]{Lemma}
\aliascntresetthe{lemma}

\newaliascnt{corollary}{theorem}  
\newtheorem{corollary}[corollary]{Corollary}
\aliascntresetthe{corollary}

\newaliascnt{claim}{theorem}  

\aliascntresetthe{claim}

\newaliascnt{conjecture}{theorem}  
\newtheorem{conjecture}[conjecture]{Conjecture}
\aliascntresetthe{conjecture}

\newaliascnt{question}{theorem}  
\newtheorem{question}[question]{Question}
\aliascntresetthe{question}

\newaliascnt{defn}{theorem}  
\newtheorem{defn}[defn]{Definition}
\aliascntresetthe{defn}

\newaliascnt{example}{theorem}  
\newtheorem{example}[example]{Example}
\aliascntresetthe{example}

\theoremstyle{remark}
\newaliascnt{remark}{theorem}  
\newtheorem{remark}[remark]{Remark}
\aliascntresetthe{remark}

\newcommand{\Spc}{\operatorname{Spc}}
\newcommand{\Shv}{\operatorname{Shv}}
\newcommand{\Psh}{\operatorname{Psh}}
\newcommand{\sPsh}{\operatorname{sPsh}}

\newcommand{\MW}{\operatorname{MW}}
\newcommand{\M}{\operatorname{M}}
\newcommand{\GW}{\operatorname{GW}}

\newcommand{\Aut}{\operatorname{Aut}}

\begin{document}
\title{Relative $\AA^1$-Contractibility of Koras-Russell prototypes and Exotic Motivic Spheres}
\author{Krishna Kumar Madhavan Vijayalakshmi}
\address{Department of Mathematics Federiques Enriques, University of Milan, Italy \& IMB UMR5584, Universit{\'e} de Bourgogne-Europe, Dijon, France }
\email{krishna.madhavan@unimi.it \& krishna.madhavan@ube.fr}
\urladdr{https://krishmv.github.io/}

% --------------------------------------------------------------
\begin{abstract}
Koras-Russell threefolds form a certain algebraic family of smooth, affine algebraic varieties that exhibit exotic behavior in the algebro-geometric context. Our goal in this paper is to extend its $\AA^1$-contractibility from a base field to a base scheme and formulate a general strategy to also extend the $\AA^1$-contractibility of its higher dimensional prototypes over a base scheme. As a major consequence, we establish the existence of exotic motivic spheres in all dimensions at least 4 over infinite perfect fields.
\end{abstract}
\maketitle

% --------------------------------------------------------------
\section{Introduction}

One of the significant open problems in algebraic geometry is to characterize the affine $n$-space $\AA^n$ among smooth affine $\AA^1$-contractible schemes of dimension $n$. This is algebro-geometric analog of the celebrated \emph{Open Poincaré Conjecture} (now a theorem: see \cite[\S 1.1]{asok2021A1} for a historical background and chronological contributions towards the proof) that characterizes $\RR^n$ among open contractible smooth $n$-manifolds using tools from topology at infinity. The analogous algebro-geometric characterization has been established in low dimensions over fields. In dimension 1, it is a well-established fact that the affine line $\AA^1_k$ is the only smooth $\AA^1$-contractible curve over any field $k$ (see \cite[Claim 5.7]{asok2007unipotent}, \cite[Proposition 3.3]{DMO}). For surfaces, the recent result of Choudhury-Roy \cite[Theorem 1.1]{choudhury2024} characterizes $\AA^2_k$ to be the unique smooth \emph{affine} $\AA^1$-contractible surface whenever $k$ is a field of characteristic zero. Recently, the authors in \cite[Theorem 3.12]{DMO} have enhanced this characterization to smooth separated surfaces. In the same vein, they also prove that affineness is, in fact, a direct consequence of $\AA^1$-contractibility for smooth surfaces (\cite[Proposition 3.13]{DMO}). However, let us remark that this characterization is widely open when the field is of positive characteristic (cf. \cite[\S 3.3]{DMO} for an overview). From dimensions $n\ge 4$, results such as \cite[Theorem 5.1, Theorem 5.3]{asok2007unipotent} disprove this characterization by constructing infinitely many pairwise non-isomorphic strictly quasi-affine smooth $n$-schemes that are all $\AA^1$-contractible.
\medskip

On one hand, the failure of this characterization for smooth affine threefolds over fields of characteristic zero is witnessed by the presence of a family of \emph{Koras-Russell threefolds} $\mcal{K}$, which are one of the main objects of study of this paper. Traditionally, this family is defined over a field of characteristic zero (say $\CC$) through the polynomial equation given by
\begin{equation}\label{intro: KR3F-threefolds-eqn}
    \KK := \{x^mz = x+y^r+t^s\}\subset \AA^4_{\CC}
\end{equation} 
where $m,r,s\geq 2$ are integers with $r$ and $s$ coprime. The counter-intuitive properties of these threefolds $\KK$ are that they are topologically contractible \cite{KR97} when seen as analytic manifolds, and are $\AA^1$-contractible in both stable $\SH(\CC)$ \cite[Theorem 4.2]{HKO16}, and unstable $\mcal{H}(\CC)$ \cite[Theorem 1.1]{DF18} $\AA^1$-homotopy theory of Morel-Voevodsky \cite{MV99}. However, they are not isomorphic to $\AA^3_{\CC}$ due to the celebrated ring-theoretic invariants such as that of Makar-Limanov \cite{makar1996hypersurface} and Derksen \cite{derksen1997constructive}. Such smooth affine $n$-dimensional varieties that are $\AA^1$-contractible but not isomorphic to $\AA^n$ are called \emph{exotic}. This terminology arises from a direct topological analogy: a smooth complex $n$-manifold is said to be \emph{exotic} if it is topologically contractible but not isomorphic to $\CC^n$. In a more generality, \cite{dubouloz2025algebraicfamilies} introduced and studied the higher dimensional analogues of these threefolds, which we call the "Koras-Russell prototypes" (see \cref{defn:GKR3F}) and proved its $\AA^1$-contractibility over a field of characteristic zero.
\medskip

On the other hand, the failure of this characterization for threefolds over fields of positive characteristics is intertwined with the solution of the classical \emph{Zariski Cancellation Problem} (ZCP), which began its formulation in the setting of commutative algebra. In algebro-geometric terms, we say that an $m$-dimensional smooth affine $k$-variety $X$ has this \emph{cancellation property} with respect to the affine space if 
\[
X \times_k \AA^1_k \cong_k \AA_k^{m+1}\ \textrm{implies that}\ X\cong_k \AA^m_k.
\]
From dimensions $\ge 3$, the failure of this cancellation over fields of positive characteristic follows due to the presence of the \emph{Asanuma-Gupta varieties} (see \cref{rem:AG-isnot-KR3F}). These were constructed in \cite{gupta2014ZCPpaper-2, gupta2014ZCPpaper-1} building on key results of \cite{asanuma1987polynomial}, thereby completely solving the state of ZCP over fields of positive characteristics (see \cite{gupta2022Zariski} for a comprehensive survey). For completeness, let us also note that ZCP holds over arbitrary fields in dimensions up to 2 (see \cite{abhyankar1972uniqueness, eakin2006cancellation} for curves and \cite{miyanishi1980affine, russell1981affine} for surfaces). As of now, ZCP stands unsolved in characteristic zero. Interestingly, \cite{dubouloz2009cylinder} proved that the cylinders over $\KK$ have the same Makar-Limanov invariant as that of $\AA^4_{\CC}$, which prevents us from concluding that $\KK\times\AA^1 \ncong_k \AA^4_k$. 
\medskip

We briefly introduce the following notions from the realm of affine algebraic geometry before elaborating on the goals of this paper.

\begin{defn}\label{def:A1-fib-space}
An \emph{$\AA^n$-fiber space} over a scheme $S$ is a smooth morphism of finite presentation $f:X \to S$ whose fibers $X_s:= X\times_S \Spec \kappa(s)$ over all points $s\in S$ are isomorphic to the affine space $\AA^n$ over the corresponding residue field $\kappa(s)$.
\end{defn}

\begin{defn}
An $\AA^n$-fiber space $f:X\to S$ as above is called a \emph{locally trivial $\AA^n$-bundle} in the Zariski (resp. Nisnevich, resp. \'etale) topology if every $s\in S$ admits a Zariski (resp. Nisnevich, resp. \'etale)  neighborhood $U$ in $S$ such that  $U\times _S X\cong U\times \AA^n$ as schemes over $U$. 
\end{defn}

%----------------------------- Main ------------------------------
We will now briefly elaborate the main results of this paper.

\subsubsection{1. Relative $\AA^1$-Contractibility of Koras-Russell Prototypes} 
We first extend the $\AA^1$-contractibility of the Koras-Russell threefolds $\KK$ over a base scheme (\cref{KR3Fmain:Noetherian}).

\begin{theorem*}
Let $S$ be a Noetherian base scheme with perfect residue fields. Then the family of smooth affine threefold $\KK \to S$ is an $\AA^1$-contractible scheme over $S$.
\end{theorem*}

Towards this result, we first prove that these smooth affine varieties are $\AA^1$-contractible over perfect fields (\cref{A1-cont-over-perfect-fields}). As a consequence, we have lifted the $\AA^1$-contractibility of Koras-Russell threefolds from characteristic zero to mixed characteristics. In succinct terms, the coprimality of the coefficients (i.e., $r$ and $s$) in the defining equation of $\KK$ allows us to produce a finite \'etale cover by choosing the coordinate axes (i.e., $y$ or $t$) whose power is coprime to characteristic of the field $k$. Thus, in case when char $k= p>0$, the \'etale locally trivial bundle in \cref{exist:algspace} can be produced by altering the axis onto which we obtain the desired finite \'etale cover depending on if $p$ divides $r$ or $s$ (for a precise formulation, see \cref{strategy: alg-space}). As a consequence, this produces the first examples of exotic arithmetic schemes over a base such as $\Spec \ZZ$ (\cref{KR3FoverZ}). We remark that these results are concise applications of the gluing theorem of Morel-Voevodsky \cite[Theorem 2.21]{MV99}, which in essence, states that a morphism $f:X \to Y$ of smooth schemes is an $\AA^1$-weak equivalence if and only if for all $y \in Y$, the induced morphism on its fibers $X\times_{Y} \Spec \kappa(y) =: X_y \to \Spec \kappa(y)$ is an $\AA^1$-weak equivalence (see \cref{fiberwise=relative}). Another interesting consequence is also that all these extended families of $\AA^1$-contractible varieties add to the list of "potential" distinct counter-examples to the ZCP in positive characteristic, thereby also producing examples of exotic affine varieties in mixed characteristic (see \cref{rem:AG-isnot-KR3F}).
\medskip

Following this, we extend the $\AA^1$-contractibility of the Koras-Russell prototypes. For concreteness, let us consider one of its simplest prototypes. The family of smooth affine varieties defined by the equation
$$\XX_m := \{\ul{x}^n z = y^r+t^s + x_0\} \subset \AA^{m+4}_k \cong  \Spec k[x_0,\dots,x_m,y,z,t] $$
is a smooth affine scheme of dimension $(m+3)$ over a field $k$, where $m \ge 0$ and $n,r,s\ge 2$ with $\text{gcd} (r,s)=1$ and $\ul{x} := \prod_{i=0}^{m} x_i$. As a higher dimensional generalization of \cref{KR3Fmain:Noetherian}, we obtain the following result (\cref{KR3F:prototypes-A1-cont-base:Noetherian}).

\begin{theorem*} The family of smooth affine variety $\XX_m \to S$ is an $\AA^1$-contractible scheme over a Noetherian base scheme $S$ with perfect residue fields.
\end{theorem*}

As in the case of threefolds, this is obtained by first extending the $\AA^1$-contractibility over perfect fields (\cref{KR3Fprototypes:field-A1-cont:perfect}), which subsequently provides us with arithmetically exotic examples over $\Spec \ZZ$ (\cref{GKR3F-A1-Cont-overZ}).

% ------------------------------------------------
\subsubsection{2. Exotic Motivic Spheres}
A solid consequence of these results is the existence of exotic motivic spheres (see \cref{sect:exotic-motivic-spheres}) in higher dimensions. Briefly, we call a motivic space to be an exotic sphere if it has the $\AA^1$-homotopy type of $\SS^{2n-1,n}$ without being isomorphic to $\AA^n\bs \{0\}$ as a scheme. Our main goal here is to exhibit that the Koras-Russell prototypes pave way for producing families of exotic motivic spheres, in all dimensions at least 3, thereby answering the question of the existence of such exotic spheres in all dimensions. The following theorem puts forth this idea into concrete terms (\cref{exotic-motspheres-countereg}).

\begin{theorem*}
Let $k$ be any infinite perfect field. Then the strictly quasi-affine schemes $\XX_m \bs\{\bullet\}$ are exotic motivic spheres over $k$. 
\end{theorem*}

To this end, we set up a cofiber sequence associated with the open immersions $\XX_m \bs\{\bullet\} \hookrightarrow \XX_m$ and $\AA_k^{m+3} \bs\{0\} \hookrightarrow \AA_k^{m+3}$ and prove that the map $\XX_m \bs \{\bullet\} \to \AA_k^{m+3} \bs\{0\}$ as described in \cref{subsec:exist-exotic} is an $\AA^1$-weak equivalence of motivic spaces. The fact that these varieties are $\AA^1$-homotopic is obtained by invoking an $\AA^1$-excision style result facilitated by the $\AA^1$-simply connectedness of $\XX_m$ and $\XX_m\bs \{\bullet\}$. The conclusion that these varieties are non-isomorphic follows from the computation of the Makar-Limanov invariants.
\medskip

Examples of motivic spaces that are $\AA^1$-homotopic to $\AA^n\bs \{0\}$, but not isomorphic to $\AA^n \bs \{0\}$ are not hard to construct (for example, it can be shown that $\SL_n / \SL_{n-1} \simeq \AA^n \bs \{0\}$, but these cannot be isomorphic for dimension reasons). However, the novelty here is that the ones produced in this paper are the first known family of examples belonging to the same dimension.

% --------------------------------------------------------
\section*{Acknowledgments}
The author wishes to express his sincere gratitude to his mentors Adrien Dubouloz and Paul Arne {\O}stv{\ae}r for their invaluable guidance and unceasing support. The author is grateful for the various research funding offered by the Università degli Studi di Milano, Italy, and Université de Bourgogne-Europe, France, for facilitating this work in a collaborative environment. The author was also a recipient of the grant "Programme Vinci C2-379" and greatly acknowledges the l’Université Franco Italienne / Università Italo Francese (UFI/UIF) for their generous support towards the mobility.

% --------------------------------------------------------------
\section*{Convention} 
Unless specified otherwise, all schemes considered are smooth and are of finite type over a base scheme. Throughout, $k$ will denote a base field and $S$ a base scheme. The category of smooth $k$-schemes (resp $S$-schemes) of finite type is denoted by $\Sm_k$ (resp. $\Sm_S$). We denote the category of presheaves (resp. sheaves) by $\Psh(\Sm_s)$ (resp. $\Shv(\Sm_S)$) and the category of motivic $S$-spaces which are simplicial presheaves by $\Spc_S:= \sPsh(\Sm_S)$. The unstable (resp. stable) $\AA^1$-homotopy category over a scheme $S$ will be denoted by $\HH(S)$ (resp. $\SH(S)$). For motivic spaces $X, Y \in \Spc_S$, its $\AA^1$-weak equivalence is denoted by $X \simeq Y$ (lives in $\mcal{H}(S))$, while its isomorphism is denoted by $X\cong Y$ (lives in $\Sm_S$).
\medskip

In \cref{sec:MHT}, we recall some basic notions of motivic homotopy theory. \cref{sec:Rel-A1-Contr-KR3F} deals with the relative $\AA^1$-contractibility of Koras-Russell threefolds over a base scheme and \cref{sec:Gen-KR3F} deals with its higher dimensional prototypes over a base scheme. In \cref{sect:exotic-motivic-spheres}, we prove the existence of exotic motivic spheres in all dimensions $\ge 4$ over infinite perfect fields.

% -------------------------------------------------------
\section{Motivic homotopy theory} \label{sec:MHT}
\subsection{Preliminaries}
Motivic homotopy (also called the $\AA^1$-homotopy) theory is a relatively nascent domain in algebraic geometry, established by Morel and Voevodsky in their seminal work \cite{MV99}. One of the conceptual aims was to provide a homotopical framework for algebraic varieties to tackle the then-open Milnor Conjecture (see \cite{orlov2007exact}) and the Bloch-Kato Conjecture (see \cite{voevodsky2003Z/2}). Soon after its formulation, it became clearer that the $\AA^1$-homotopy category provides a natural ground to several of Grothendieck's visions towards the universal theory of motives. The foundation of the $\AA^1$-homotopy category is facilitated by the homotopy theory of the simplicial presheaves, equivalently, the category of \emph{motivic $S$-spaces} $\Spc_S:= \sPsh(\Sm_S)$ equipped with the discrete simplicial structure. The category $\Spc_S$ is then subjected to two formal Bousfield localizations, thereby producing the desired $\AA^1$-homotopy category $\mcal{H}(S)$. Roughly speaking, the first localization inverts the affine line $\AA^1$ by a process of \emph{$\AA^1$-localization}, thereby producing motivic spaces that are $\AA^1$-invariant; in other words, a motivic space $X \to S$ is \emph{$\AA^1$-invariant} if the induced morphism
    $$\pr_1^*: \Hom_{\Spc_S}(U,X)\to \Hom_{\Spc_S}(\AA^1_U, X)$$
is a bijection for the canonical projection $\pr_1:\AA^1_U\to U$, for every $U\in \Sm_S$. The $\AA^1$-localization functor $\L_{\AA^1}: \Psh(\Sm_S) \to \Psh_{\AA^1}(\Sm_S)$ is the left adjoint to the inclusion functor $\iota: \Psh_{\AA^1}(\Sm_S) \hookrightarrow \Psh(\Sm_S)$ of presheaves that are $\AA^1$-invariant. The second localization inverts the descent maps into weak equivalences; that is, a simplicial presheaf $X\in \Spc_S$ satisfies \emph{Nisnevich descent} if for every \v{C}ech covers $\check{U}_{\bullet}\to X$, we have that the ${\mathrm{hocolim}_n\ U_n \xrightarrow{\simeq} X}$ is a weak equivalence of simplicial presheaves. The Nisnevich localization inverts these maps in the $\AA^1$-homotopy category via the functor $\L_{\Nis}: \Psh(\Sm_S)\to \L_{\Nis} \Psh(\Sm_S)$ which is the left derived functor of the identity functor. Although in a general context, one replaces the \v{C}ech covers by hypercovers and studies the Nisnevich hyperdescent (see \cite[\S 3.4]{antieau2017primer}). 
\medskip

The \emph{$\AA^1$-homotopy category} is then defined as the subcategory of $\sPsh(\Sm_S)$ consisting of simplicial presheaves those that are $\AA^1$-invariant $\sPsh(\Sm_S)_{\AA^1}$ and those that satisfy the Nisnevich descent $\Shv_{\Nis}(\Sm_S)$ 
    $$\mcal{H}(S):= \sPsh(\Sm_S)_{\AA^1} \cap \Shv_{\Nis}(\Sm_S) \subset \sPsh(\Sm_S)=:\Spc_S. $$
However, note that for a simplicial presheaf $\mcal{F}$, $\L_{\AA^1}( \L_{\Nis}(\mcal{F}))$ need not be a Nisnevich sheaf and $\L_{\Nis}(\L_{\AA^1}(\mcal{F}))$ need not be $\AA^1$-invariant (cf. \cite[\S 3, Example 2.7]{MV99}). This issue is circumvented by considering the \emph{motivic localization functor} $\L_{mot}: \Psh(\Sm_S) \to \mcal{H}(S)$ which is the left adjoint of the inclusion $\iota: \mcal{H}(S) \hookrightarrow \Psh(\Sm_S)$ defined as 
$$\L_{mot}:= \textrm{colim} (\L_{\Nis}\to \L_{\AA^1}\L_{\Nis} \to \L_{\AA^1}\L_{\Nis}\L_{\AA^1}\to \dots..) =  (\L_{\AA^1}\circ \L_{\Nis})^{\circ \mathbb{N}}.$$
Being a left adjoint, it preserves colimits. Furthermore, it is locally cartesian and preserves finite products (\cite[Proposition 3.15 \& 4.1]{hoyois2017six}). For a detailed background on the construction of the motivic homotopy category, we redirect the reader to \cite{hoyois2017six, antieau2017primer, asok2021A1}. Let us now recollect some basic jargon in the $\AA^1$-homotopy category.
\begin{defn}\label{defn:A1-conn}
The \emph{$\AA^1$-connected component sheaf} of a scheme $X\to S$ is the Nisnevich sheaf associated to the presheaf 
        $$U\in \Sm_S \mapsto \Hom_{\mcal{H}(S)}(U,X) $$
for all $U\in \Sm_S$. The $\AA^1$-connected component sheaf of $X$ is denoted by $\pi_0^{\AA^1}(X)$ and is a Nisnevich sheaf of sets on $\Sm_S$. We say that a scheme $X$ is \emph{$\AA^1$-connected} if $\pi_0^{\AA^1}(X)$ is isomorphic to the trivial sheaf $S$.
\end{defn}
\begin{example}\label{S-point}
Every $\AA^1$-connected smooth scheme $X$  over a locally Henselian scheme $S$ admits an $S$-point (\cite{MV99}, Remark 2.5). In other words, any variety that does not possess a rational point cannot be $\AA^1$-connected.
\end{example}
\begin{defn}
A $k$-scheme $X$ is \emph{$\AA^1$-rigid} if the induced morphism         
    $$\pr_1^*: \Hom_{\Sm_S}(U,X)\to \Hom_{\Sm_S}(U\times \AA^1, X)$$
is a bijection, for the canonical projection morphism $\pr_1:\AA^1_U \to U$, for every $U\in \Sm_S$.
\end{defn}
\begin{example}
Any 0-dimensional smooth scheme over a field is $\AA^1$-rigid (\cite[Example 2.1.10]{asok2021A1}). The affine variety  $\GG_m$ is $\AA^1$-rigid (\cite[Example 4.1.10]{asok2021A1}). Any smooth projective curve of positive genus and any Abelian variety are $\AA^1$-rigid varieties. More examples can be found in \cite[\S 4.1]{asok2021A1} and \cite[\S 2]{choudhury2024}.
\end{example}
As in topology, one defines the analogous notion of path connectedness in algebraic geometry via images from the affine line $\AA^1$ replacing the unit interval $[0,1]$.
\begin{defn}\label{defn:A1-chainconn}
A $k$-scheme $X$ is \emph{$\AA^1$-chain connected} if for any finitely generated separable field extension $L/k$ and for all rational points $x$ and $y \in X(L)$, there exists rational points $x= x_0,\dots, x_{n-1}$, $x_n =y\in X(L)$ and an \emph{elementary $\AA^1$-weak equivalence} $f_i: \AA^1_L \to X$ between $x_{i-1}$ and $x_i$, that is, $f_i(0)= x_{i-1}$ and $f_i(1)= x_i$ for $1\le i\le n$.
\end{defn}

\begin{example}
Affine $n$-spaces or more generally, any $n$-dimensional smooth $k$-variety covered by affine spaces in the sense of \cite[Definition 2.2.10]{asokmorel2011} are $\AA^1$-chain connected. Any $\AA^1$-chain connected variety is $\AA^1$-connected (\cite[Proposition 2.2.7]{asokmorel2011}) but the converse is not true (\cite[\S 4.1]{balwe2015a1}). The smooth exotic surfaces such as Ramanujam surface and tom Dieck-Petrie surfaces cannot be $\AA^1$-chain connected as they are not even $\AA^1$-connected (this is a consequence of the characterization of $\AA^2$ (\cite[Theorem 1.1]{choudhury2024}) over a field of characteristic zero.
\end{example}

\begin{defn}
A motivic $S$-space $\XX$ is said to be \emph{$\AA^1$-contractible} if the canonical map $\XX \to S$ is an $\AA^1$-weak equivalence in $\Spc_S$ or equivalently, an isomorphism in $\mcal{H}(S)$.    
\end{defn}

\begin{example}\label{A1-contr:egs}
By construction, the affine line $\AA^1$ is an $\AA^1$-contractible curve. More generally, affine $n$-spaces $\AA^n$ form the prototypical examples of $\AA^1$-contractible varieties by inducting on $n$. In fact, the usual radial rescaling morphism
	$$\AA^n\times \AA^1\to \AA^n \quad (x_1\dots,x_n,t)\mapsto (tx_1,\dots,tx_n)$$ 
is an $\AA^1$-weak equivalence of motivic spaces. The singular cuspidal curves $C_{r,s}$ defined by the equation $\{y^r-t^s= 0 \}$, for coprime integers $r,s$ are $\AA^1$-contractible (\cite[Example 2.1]{asok2007unipotent}). In fact, the normalization map
        $$\AA^1 \to C_{r,s} \quad w \mapsto (w^s, w^r) $$
is an $\AA^1$-weak equivalence of motivic spaces. Any vector bundle morphism is, in fact, an $\AA^1$-homotopy equivalence (\cite[Example 2.2]{MV99}). More generally, any Zariski (resp. Nisnevich) locally trivial bundles with $\AA^1$-contractible fibers are $\AA^1$-weak equivalences.
\end{example}

% -----------------------------------------------------------
\subsection{Functoriality in Motivic Homotopy Theory}
We shall discuss the functoriality in the motivic homotopy theory with a special emphasis on the base change property that is crucial for our context. For a versatile background on functoriality, we refer to \cite[\S 3]{MV99}, \cite{Ayoub2007six, CD2019}. For other recent accounts, see also e.g., \cite[\S 4]{hoyois2017six}, \cite[\S 2]{ELSO22}, and \cite[\S 3]{rondigs2025grothendieck}. For a comprehensive overview of functoriality in the relative setting, the readers are redirected to \cite[\S 1.2]{DMO} from which we will now outsource some properties relevant for this paper.

\begin{prop} \label{cor:pushpull-A1-weak} 
Let $f :T\to S$ be a morphism of schemes. Then the following holds:
\begin{enumerate}
\item For every morphism $u:Y\to X$ between smooth $S$-schemes which is an $\AA^1$-weak equivalence in $\Spc_S$, the morphism $u_T: Y_T\to X_T$ is an $\AA^1$-weak equivalence in $\Spc_T$.  
\item If $f: T\to S$ is smooth, then every morphism $u: Y\to X$ between smooth $T$-schemes which an $\AA^1$-weak equivalence in $\Spc_T$ is also an $\AA^1$-weak equivalence in $\Spc_S$.
\end{enumerate}
\end{prop}
\begin{proof}
See \cite[Proposition 1.7]{DMO}.
\end{proof}

\begin{prop} \label{pullbackfunctor} 
Let $u: X \to S$ be a smooth $\AA^1$-contractible $S$-scheme. Then for any arbitrary morphism of schemes $f : T \to S$, the induced map $u_T: X_T\to T$ is a smooth $\AA^1$-contractible $T$-scheme, where we have denoted the scheme-theoretic pullback by $X_T:= X\times_S T$.
\end{prop} 
\begin{proof}
This is a straightforward application of \Cref{cor:pushpull-A1-weak} (1).
\end{proof}

\begin{corollary}\label{basechange:imperfect} \label{cor:A1-cont-fibers} \label{basechange:closedimm}
The \cref{pullbackfunctor} has several interesting implications.
\begin{enumerate}
\item Let $X$ be a smooth $\AA^1$-contractible scheme over a field $k$. Then for every field extension $k\subset L$ (separable or not), $X_L:=X\times_{\Spec k} \Spec L$ is an $\AA^1$-contractible $L$-scheme.
\item Let $f: X \to S$ be a smooth $\AA^1$-contractible $S$-scheme. Then for every point $s$ of $S$ (closed or not) with residue field $\kappa(s)$, the scheme theoretic fiber $X_s:=X\times_{S} \Spec \kappa(s)$ of $f$ over the point $s$ is a smooth $\mathbb{A}^1$-contractible $\kappa(s)$-scheme.
\item Let $i: Z \hookrightarrow S$ be a closed immersion of schemes. If $f:X\to S$ is a smooth $\AA^1$-contractible $S$-scheme, then $\mathrm{pr}_2: X\times_S Z \to Z$ is an $\AA^1$-contractible $Z$-scheme.
\end{enumerate}
\end{corollary}

\begin{defn} \label{intro:rel-A1-contr} \label{subsec:Base-change}
A morphism of $S$-schemes $\phi: X \to Y$ is a \emph{relative $\AA^1$-weak equivalence} if $\phi$ is an $\AA^1$-weak equivalence in $\Spc_Y$ as motivic $Y$-spaces.
\end{defn}

In particular, for smooth morphisms we have the following:
\begin{corollary}\label{lem:3-from-2}
Let $f: Y\to X$ and $g:Z\to Y$ be smooth morphisms between smooth schemes over a scheme $S$. If $g:Z\to Y$ is an $\AA^1$-weak equivalence in $\Spc_Y$ and $f:Y\to X$ is an $\AA^1$-weak equivalence in $\Spc_X$, then $h= f\circ g:Z\to X$ is an $\AA^1$-weak equivalence in $\Spc_X$.
\end{corollary}
\begin{proof} 
See \cite[Corollary 1.8]{DMO}.
\end{proof}

On the other hand, the implication "\emph{$f$ and $g$ are $\AA^1$-weak equivalence in $\Spc_X$ $\Rightarrow$ $g$ is an $\AA^1$-weak equivalence in $\Spc_Y$}" in \Cref{lem:3-from-2} does not hold in general (see \cite[Example 1.9]{DMO}). This shows that being "relatively" $\AA^1$-contractible is a subtle notion. Crucially to our interest, the following result shows that the relative $\AA^1$-contractibility behaves nicely in families for smooth schemes.

\begin{prop}\label{fiberwise=relative} \label{pointwise:phenomenon}
For a scheme $f: X \to S$ smooth over a qcqs scheme $S$ of finite Krull dimension, the following properties are equivalent:
\begin{enumerate}
\item The morphism $f:X\to S$ is an $\AA^1$-weak-equivalence in $\Spc_S$.
\item For every point $s$ of $S$ with residue field $\kappa(s)$, the scheme theoretic fiber of $f$ over $s$ is a smooth $\mathbb{A}^1$-contractible $\kappa(s)$-scheme.
\end{enumerate}
\end{prop}
\begin{proof} 
See \cite[Proposition 1.15]{DMO}.
\end{proof}

%%----------------------------------------------------------
\subsection{Milnor-Witt $K$-Theory}
To prove the $\AA^1$-contractibility of Koras-Russell threefolds over fields, we will need a description of the endomorphism ring of the motivic spheres $\AA^n\bs \{0\}$. The Milnor-Witt $K$-theory is essentially an extension of Milnor $K$-theory, which cleverly mixes the generators and relations of Milnor $K$-theory and Grothendieck-Witt ring. Interested readers can refer to \cite{morel2012A1topology, fasel2020Chow-Witt, deglise2023notes, bachmann2025MWmotives} for an extensive background.

\begin{defn}
Let $F$ be any arbitrary field and $F^{\times}$ be its units. The \emph{Milnor-Witt $K$-theory} or \emph{Milnor-Witt ring}, denoted as $K^{\MW}_*(F)$, is defined as the $\ZZ$-graded associative unital ring freely generated by the formal symbols $[a]$ of degree +1, for $a\in F^{\times}$ and a symbol $\eta$ of degree -1 (called the \emph{Hopf element}) subjected to the following relations:

\begin{enumerate}\label{MW:relations}
\item $[a][1-a] = 0$, for any $a\ne 0,1$ (Steinberg relations)
\item $[ab] = [a]+[b]+\eta[a][b]$, for any $a,b\in F^{\times}$
\item $\eta[a]=[a]\eta$, for any $a\in F^{\times}$
\item $\eta(\eta[-1]+2) = 0$
\end{enumerate}
\end{defn}
It is related to the Milnor $K$-theory \cite{milnor1970algebraic} by killing $\eta$, that is, $K^{\MW}_*/(\eta) = K_*^{\M}(F)$. It is related to the Grothendieck-Witt ring of symmetric bilinear forms by killing the hyperbolic form $h:=\Lin -1,1\Rin $, that is, $K^{\MW}_{*}{F}/(h) = \GW(F)$. There are various standard relations that one immediately derives by working with the generator and the relations. One such important element is the $\epsilon$ in degree 0 described as follows:
\begin{lemma}\label{MW:epsilon-defn}
Let $\epsilon:= -\Lin -1 \Rin \in K^{\MW}_0(F)$. For any $n\in \ZZ$, let 
\[ n_{\epsilon} = 
\begin{cases}
\sum_{i=1}^{n} \Lin (-1)^{i-1} \Rin           & \text{if}\quad n> 0.  \\
0                                               & \text{if}\quad n = 0.  \\
\epsilon \sum_{i=1}^{-n} \Lin (-1)^{i-1} \Rin   & \text{if}\quad n<0
\end{cases} 
\]
Then the following holds.
\begin{enumerate}
\item For any $\alpha\in K^{\MW}_n(F)$ and $\beta\in K^{\MW}_m(F)$, we have $\alpha\beta = \epsilon ^{mn}\beta\alpha$
\item For any $n\in \ZZ$ and $a\in F^\times$, we have $[a^n] = n_{\epsilon}[a]$.
\end{enumerate}
\end{lemma}
\subsection*{Residue homomorphism:} One of the fundamental tools to build a Gersten-type (cochain) complex with coefficients in the Milnor-Witt $K$-theory is the \emph{residue homomorphism}. Let $F$ be a field and $v: F \to \ZZ\cup\{-\infty\}$ be a discrete valuation with residue field $\kappa(v)$ and valuation ring $\mathcal{O}_v$ and choose a uniformizing parameter $\pi = \pi_v$ of $v$.
\begin{theorem}
There is a unique homomorphism of graded Abelian groups 
        $$\partial^{\pi}_v : K^{\MW}_*(F)\to K^{\MW}_{*-1}(\kappa(v))$$
of degree -1 such that $\partial_v^\pi$ commutes wit the multiplication by $\eta$ and 
\begin{enumerate}
\item $\partial_v^\pi([\pi, u_1,\dots,u_n]) = [\Bar{u_1},\dots,\Bar{u_n}]$
\item $\partial_v^\pi([u_1,\dots,u_n])=0$
\end{enumerate}
\end{theorem}
We call such a $\partial_v^\pi$ the \emph{residue homomorphism}. It turns out that the so-defined homomorphism depends on the chosen valuation $\pi$, that is, for an element $\alpha\in K^{\MW}_*(F)$ and $u\in \mathcal{O}_v^{\times}$, we have that 
        $$\partial_v^\pi(\Lin u \Rin \alpha) =  \Lin \Bar{u}\Rin \partial_v^\pi(\alpha).$$
This is the reason for \cite{morel2012A1topology} to introduce the theory of \emph{twisted Milnor Witt $K$-theory}. The twists are considered with respect to an invertible sheaf, equivalently, a line bundle\footnote{in fact, one uses the graded line bundles in more generality, see \cite[\S 1.3]{fasel2020Chow-Witt}}. Let $V$ be an $F$-vector space of rank 1 with $V^\times$ comprising the non-zero elements of $V$. The group $V^\times$ acts transitively on $F^\times$, and this gives the free Abelian group $\ZZ[V^\times]$ a structure of a $\ZZ[F^\times]$-algebra. On the other hand, we know that $K^{\MW}_*(F)$ is a $K^{\MW}_0(F)$-algebra and there is a canonical map $F^\times \to K^{\MW}_0(F)$ define as $u \mapsto
\Lin u \Rin\in K^{\MW}_0(F)$. Thus, all the groups $K^{\MW}_n(F)$ also inherits the structure of $\ZZ[F^\times]$-algebra. Now, we can conveniently set
the \emph{Milnor-Witt $K$-theory twisted by $V$} as        
            $$K^{\MW}_n(F,V) := K^{\MW}_n(F)\otimes_{\ZZ[{F^\times}]} \ZZ[V^\times]$$
With this definition, we can now consider the twisted residue homomorphism with respect to a line bundle $L$. The homomorphism
        $$\partial_v: K^{\MW}_*(F,L)\to K^{\MW}_{*-1}(\kappa(v), (\mathfrak{m}_v/\mathfrak{m}^2_v)^*\otimes L)$$
is defined by $\partial_v(\alpha \otimes l) = \partial_v^\pi(\alpha)\otimes \overline{\pi}^* \otimes l$, for $l\in L$. Here, $(\mathfrak{m}_v/\mathfrak{m}^2_v)^*$ represents the relative tangent space which, by definition, is the dual of the relative cotangent space $\mathfrak{m}_v/\mathfrak{m}^2_v$ (seen as $\kappa(v)$-vector spaces).
\begin{lemma}
The homomorphism $\partial_v$ is well-defined and is independent of the uniformizer $\pi$.
\end{lemma}
\begin{proof}
See \cite[Lemma 1.19]{fasel2020Chow-Witt}.
\end{proof}
% --------------------------------------------------
\subsection*{Motivic Brouwer degree:} 
A fundamental tool in the Milnor-Witt $K$-theory which gives the description of the non-trivial $\AA^1$-homotopy sheaf of the motivic sphere $\AA^n \bs \{0\}$ is the motivic Brouwer degree map due to Morel. This is a motivic analogue of the classical Brouwer degree map in topology.	

\begin{theorem}\label{thm:A1-Brouwer} (\cite[Theorem 6.40]{morel2012A1topology})
Let $k$ be a perfect field. Then for $n\ge 2$, we have a canonical isomorphism of strictly $\AA^1$-invariant sheaves 
    $$\pi^{\AA^1}_{n-1}(\AA^n_k\bs\{0\}) \cong \pi_n^{\AA^1}((\PP^1_k)^{\wedge n}) \cong K^{\MW}_n (k) $$
\end{theorem}
This theorem helps one to understand the endomorphism ring $[\AA^n_k \bs\{0\}, \AA^n_k\bs \{0\}]_{\AA^1}$, for all $n\ge 2$. In particular, \cite[Corollary 6.43]{morel2012A1topology} helps us to identify the zeroth part of Milnor-Witt sheaves with the Grothendieck-Witt ring of quadratic forms over a perfect field $k$:
\begin{equation}\label{Gw=K_0^MW:eqn}
    [\AA^n_k\bs\{0\}, \AA^n_k\bs\{0\}]_{\AA^1} \cong K^{\MW}_0(k) \cong \GW(k).
\end{equation}        
If $f:\Absn \to \Absn $ is a morphism of spaces, then its class of $[f]$ in $\mcal{H}(k)$ is called the \emph{motivic Brouwer degree of $f$}. To compute this Brouwer degree, one uses the following cohomological technique.
\begin{lemma}\label{A1degreemap}
Let $f:\Absn\to \Absn$ be a morphism in $\mathcal{H}(k)$. Then $f$ is an isomorphism if and only if 
        $$f^*: H^{n-1}(\Absn, K^{\MW}_n)\to H^{n-1}(\Absn, K^{\MW}_n)$$
is an isomorphism of Milnor-Witt $K$-theory groups.
\end{lemma}
\begin{proof}
See \cite[Lemma 2.2]{DF18}.
\end{proof}

% ------------------------------------------------------
\section{Relative \texorpdfstring{$\mathbb{A}^1$}{A1}-contractibility of Koras-Russell threefolds} \label{sec:Rel-A1-Contr-KR3F}

\begin{defn}\label{defn:KR3F}
For a field $k$, the \emph{Koras-Russell threefolds of first kind} are defined by the explicit polynomial equation 
\begin{equation}\label{defining-eqn:KR3F}
         \KK:= \{x^nz-y^r-t^s-x=0 \}\subset \AA^4_k
\end{equation}
where the integers $n\geq 2$ and $r,s\ge 1$ such that gcd $(r,s) = 1$.
\end{defn}
Traditionally, Koras-Russell threefolds are studied extensively over a field of characteristic zero. But in the light of coprimality of the integers $r,s$ and the usefulness of the coordinate "$x$" to make the Jacobian criterion for smoothness work, we show here that it is indeed well-defined over any field and more so over Noetherian base schemes. Note that the induced projection map $f:\KK \to \AA^1_x$ has all fibers except 0 isomorphic to $\AA^2$ and the fiber at the origin isomorphic to a cylinder on the cuspidal curve $\AA^1_x \times C_{r,s}$ (see \cref{Book-surface}). More generally, one studies a deformed version of the Koras-Russell threefolds given by $\{x^mz = x q(x)+y^r+t^s\}$, for any polynomial $q(x)$ such that $q(0)\in k^\times$ (see e.g., \cite{dubouloz2011noncancellation}). In \cite{DF18}, the authors prove the $\AA^1$-contractibility of the Koras-Russell threefold of first kind over a field of characteristic zero. In the first part of this paper, we prove that this, indeed, can be extended over general base schemes. Let us begin by putting forth the strategy in place.
\begin{figure}[h]
    \centering
    \includegraphics[width=8cm, height=4cm]{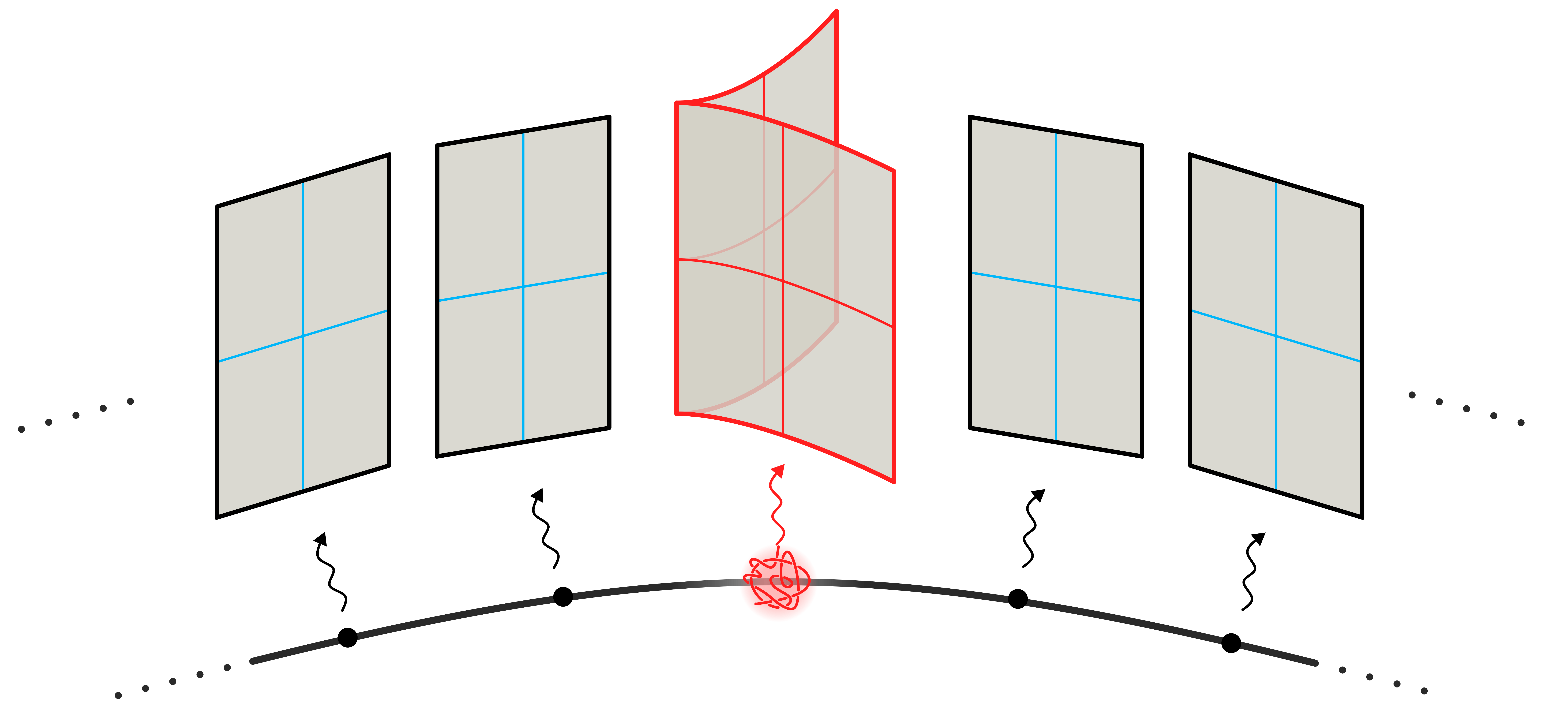}
    \caption{Fibers of the Koras-Russell threefolds projected to $\AA^1_x$ with the origin indicated in red.}
    \label{Book-surface}
\end{figure}
The basic ingredient is a version of the five lemma in the context of pointed model categories. 
\begin{lemma}\label{weak5lemma}
Let $\mcal{C}$ be any pointed model category and consider the following commutative diagram in $\mcal{C}$ where the rows are cofiber sequences.
\[ \begin{tikzcd}
	A & B & C \\
	{A'} & {B'} & {C'}
	\arrow["u", from=1-1, to=1-2]
	\arrow[swap, "f", from=1-1, to=2-1]
	\arrow["v", from=1-2, to=1-3]
	\arrow["g"', from=1-2, to=2-2]
	\arrow["h", from=1-3, to=2-3]
	\arrow[swap, "{u'}", from=2-1, to=2-2]
	\arrow[swap, "{v'}", from=2-2, to=2-3]
\end{tikzcd} \]
If $f$ and $g$ are weak equivalences, then so is $h$. In addition, assume that $f$ and $h$ are weak equivalences and $B$ is contractible in $\mcal{C}$, then $g$ is a weak equivalence in $\mcal{C}$ and consequently, $B'$ is contractible in $\mcal{C}$.
\end{lemma}
\begin{proof}
A proof is provided in \cite[Proposition 6.5.3 (b)]{hovey2007model} and is famously known as the \emph{very weak five lemma}.
\end{proof}
To adapt this lemma to our situation, we note that $\KK$ contains a hypersurface given by $\{z=0\}:= P\subset \KK$ which can be shown to be isomorphic to $\AA^2$ with coordinates $y$ and $t$, i.e., $P \cong \Spec k[y,t]$. Moreover, $\KK$ also contains an affine line given by $L:= \{x=y=t=0\} \cong \AA^1$, i.e., $L \cong \Spec(k[z])$, which gives us $P\backslash L \simeq \AA^2\backslash \{0\}$. By definition, the varieties $L$ and $P$ intersect transversally in $\AA^4_k$ at the point $(0,0,0,0)$ giving us the following pullback diagram:
\begin{equation}
\begin{tikzcd}[column sep=2pt]
    &&&&& {P\backslash L \simeq \mathbb{A}^2\backslash \{0\}} & {} &       {P\simeq \mathbb{A}^2} \\
    \\{} &&&&& {\mathcal{K}\backslash L} && {\mathcal{K}}
    \arrow[from=1-6, to=1-8]
    \arrow[from=1-6, to=3-6]
    \arrow[from=1-8, to=3-8]
    \arrow[from=3-6, to=3-8]
\end{tikzcd}
\end{equation}
Consider the associated cofiber sequence given by the following commutative diagram
\begin{equation}\label{setup:diagram1}
\begin{tikzcd}[column sep=15pt]
	&&&&& {\mathbb{A}^2\backslash \{0\}} & {} & {\mathbb{A}^2} && {\mathbb{A}^2/\mathbb{A}^2\backslash\{0\}} \\ \\
	{} &&&&& {\mathcal{K}\backslash L} && {\mathcal{K}} && {\mathcal{K}/\mathcal{K}\backslash{L}}
	\arrow[from=1-6, to=1-8]
	\arrow["i"', from=1-6, to=3-6]
	\arrow[from=1-8, to=1-10]
	\arrow["\phi", from=1-8, to=3-8]
	\arrow["j", from=1-10, to=3-10]
	\arrow[from=3-6, to=3-8]
	\arrow[from=3-8, to=3-10]
\end{tikzcd}
\end{equation}
In the light of \cref{weak5lemma}, it suffices to show that the morphisms $i$ and $j$ are $\AA^1$-weak equivalences in $\Spc_k$. 
%--------------------------------------------------
\subsection{$\AA^1$-weak equivalence of $j$:} \label{A1-weakeq-of-j}
The arguments to prove the $\AA^1$-weak equivalence of $j$ are exactly as those of \cite{DF18}, which is a consequence of the motivic purity isomorphism (\cite[Theorem 2.23]{MV99}). In fact, the defining equation of $\KK$ gives us a description of the normal cone of $L$ inside $\KK$. Namely, by rewriting $\KK$ as $\{x(x^{m-1 }z-1) = y^r+t^s\}$, one observes that the regular function ($x^{m-1}z-1$) is invertible in $L$ (as it is equal to -1 in $L$) and that the normal cone $N_{\KK}(L)$ is generated by $y$ and $t$. Coupling this data with the homotopy purity, we get 
    $$\KK /(\KK \backslash{L}) \simeq \Th(N_{\KK}(L))\simeq L_+\wedge (\PP^1)^{\wedge2}.$$
Again by purity, the quotient space $\AA^2/\AA^2\backslash \{0\}$ is identified with $(\PP^1)^{\wedge2}$ (cf. \cite[\S 4.6]{antieau2017primer} or \cite[\S 2.1]{ADF2017smooth}). The inclusion $\Spec k \simeq \{0\}\hookrightarrow L\simeq \AA^1$ induces a the composite 
$$\Spec k \wedge (\PP^1)^{\wedge 2}\simeq \AA^2/\AA^2\backslash \{0\} \xrightarrow{j} \KK/\KK\backslash L \simeq L_+\wedge (\PP^1)^{\wedge 2}$$
which is an $\AA^1$-weak equivalence (cf. \cite[Lemma 2.1]{voevodsky2003Z/2}) and so is $j$.

% --------------------------------------------------------
\subsection{$\AA^1$-weak equivalence of $i$}\label{A1-weakeq-of-i}
To prove the $\AA^1$-weak equivalence of $i:\AA^2\backslash \{0\} \to \KK\bs L$, the strategy is to first produce an explicit $\AA^1$-weak equivalence $g:\KK\bs L\to \AA^2\bs \{0\}$. Then, using a key property of motivic Brouwer degree as in \cref{A1degreemap}, we show that the composite $g\circ i: \AA^2\bs \{0\}\to \AA^2\bs \{0\}$ is an $\AA^1$-weak equivalence in $\Spc_k$.

\subsection*{The explicit $\AA^1$-weak equivalence via Zariski local triviality}
In what follows, we shall explain the strategy to produce the explicit $\AA^1$-weak equivalence $g: \KK\bs L\to \AA^2\bs \{0\}$ that has the structure of a Zariski locally trivial $\AA^1$-bundle over $\AA^2\bs\{0\}$ and the necessary adaptations when $k$ is arbitrary. The smooth affine scheme $\KK(s):=\{x^mz=x+y^r+t^s\}$ for $s=1$ gives $\KK(1):= \{x^mz = x+y^r+t\}\cong \AA^3_k$, which by the change of co-ordinates gives the equivalence $\KK(1)\bs L \simeq \AA^2\bs \{0\}$. Now the strategy is to demonstrate the existence of a quasi-affine fourfold $W$ that is the total space of a Zariski locally trivial $\AA^1$-bundle over $\mcal{K}(s)\bs L$, for all $s\geq 1$. This fourfold $W$ is produced by exploiting the Danielewski fiber product trick that was initially constructed to produce a counterexample to the cancellation problem for surfaces (\cite{danielewski1989cancellation}, see also \cite{dubouloz2007addDanielewski}). Here, we show that the same strategy can be extended over any field with suitable modifications.

\begin{prop}\label{Daniel:trick 4fold}
For any field $k$ and every $s\geq 1$, there exists a smooth strictly quasi-affine fourfold $W$ which is simultaneously the total space of a Zariski locally trivial $\AA^1$-bundles over $\KK(s)\bs L$. 
\end{prop}
\begin{proof}
The desired fourfold $W$ is constructed as the pullback of two strictly quasi-affine schemes over an algebraic space $\frak{S}$ 
\[\begin{tikzcd}
	W & {\KK(s)\backslash L} \\
	{\KK(1)\backslash L} & {\mathfrak{S}}
	\arrow["{\pr_s}", from=1-1, to=1-2]
	\arrow["{\pr_1}"', from=1-1, to=2-1]
	\arrow[from=1-2, to=2-2]
	\arrow[from=2-1, to=2-2]
\end{tikzcd}\]
whose existence is proven in \cref{exist:algspace} where we also show that for all $s\geq 1$, the quasi-affine threefolds $\mcal{K}(s)$ all have a structure of an \'etale locally trivial $\AA^1$-bundle $\rho:\mcal{K}(s) \to \frak{S}$ fibered over the same algebraic space $\frak{S}$. Apriori, by taking the fiber product $W:= (\mcal{K}(1)\bs L) \times_{\frak{S}} (\mcal{K}(s)\bs L)$, we see that $W$ only has the structure of an algebraic space. But by noticing that the natural projections $\pr_s: W\to \mcal{K}(s)\bs L$ is an affine morphism and $\mcal{K} \bs L$ is a strictly quasi-affine scheme, which consequently implies that $W$ also has to be a strictly quasi-affine scheme. The fact that this \'etale equivalence descends also to the equivalence in the Zariski topology is due to the deep fact that the transition group $\Aut(\AA^1_k) = \GG_{m,k} \ltimes \GG_a$ is special, in the sense of \cite[\S 5]{Grothendieck1958torsion}.
\end{proof}

As a direct consequence, we obtain our desired result.
\begin{corollary}\label{g-is-an-A1-weq}
The map $g: \KK(s)\bs L \to \AA^2\bs \{0\}$ is an $\AA^1$-weak equivalence for all $s\geq 1$.
\end{corollary}
\begin{proof}
This is due to the observation that Zariski locally trivial $\AA^n$-bundles preserve $\AA^1$-weak equivalences (see, for example \cite[Lemma 3.1.3]{asok2021A1}).
\end{proof}

% -------------------------------------------------------------
\subsection*{Strategy to construct $\frak{S}$ over arbitrary fields:} \label{strategy: alg-space}

The existence of the claimed algebraic space $\frak{S}$ is due to the fundamental fact in moduli theory that if an algebraic group $G$ acts freely on a scheme $X$, then the quotient $X/G$ exists in the category of algebraic spaces (\cite[\href{https://stacks.math.columbia.edu/tag/071R}{Tag 071R}]{stacks-project}). When the char $k=0$, producing a $\GG_a$-action on an affine variety $X$ is equivalent to providing a locally nilpotent derivation $D$ on its coordinate ring $\Gamma(X)$ (this is due to the bijection of the exponential map $t\mapsto \exp (tD)$ cf. \cite[Chapter 1, \S 1.5]{freudenburg2006algebraic}). However, if the char $k= p >0$, this correspondence does not apply a priori. For example, $1/p!$ in the exponential function is not well-defined. Nevertheless, one still gets a correspondence by considering locally finite iterative higher derivations $\partial$ instead of a single locally nilpotent derivation. The correspondence is then obtained by considering the truncated version of $\exp(t\partial)$ (see \cite[Part I, \S 1]{miyanishi1978lectures}, \cite[\S 5]{okuda2004kernels}). Given the existence of such a family of locally finite iterative higher derivations, we have that for a $\GG_a$-action on $\KK(s)$ its algebraic quotient coincides with the projection $\pr_{x,t}:\AA^4\to \AA^2$ restricted to ${\pr_{x,t}}\vert_{\mathcal{K}(s)}=: q: \KK(s)\to \AA^2$ with the fixed point locus $L$. Hence, this $\GG_a$-action induces a fixed point free $\GG_a$-action on $\KK(s)\bs L$ 
\[\begin{tikzcd}
	{\mathcal{K}(s)\backslash L} && {\mathbb{A}^2\backslash \{0\}}\\
	& {(\KK(s)\bs L)/\GG_a}
	\arrow["q",from=1-1, to=1-3]
	\arrow[from=1-1, to=2-2]
	\arrow[from=2-2, to=1-3]
\end{tikzcd}\]
such that it factors through an \'etale $\GG_a$-torsor over its geometric quotient $(\KK(s)\bs L)/ \GG_a$ as above. The desired algebraic space $\frak{S}$ is then constructed by an \'etale equivalence relation. Observe that for the morphism $q: \mcal{K}\bs L\to \AA^2\bs \{0\}$, the fibers over points $\{x \neq 0\}$ forms a Zariski locally trivial $\AA^1$-bundle, which is actually a trivial $\AA^1$-bundle. This is due to the fact that the defining equation of $\KK(s)$ can be rewritten as $z = x^{-m}(y^r+t^s+x)$, which gives a trivialization of the fiber over $\{x\ne 0\}$. On the other hand, the fiber over the punctured affine line $C_0:= \{x=0\} = \Spec k[t^{\pm 1}] $ is given by $C\times \Spec k[z]$ where $C:=\Spec k[t^{\pm 1},y]/(y^r +t^s)$. The curve $C \to C_0$ can be identified with a finite \'etale cover of degree $r$ with respect to the projection $(y,t) \mapsto t$. This is where the case over arbitrary fields crucially differs from that of characteristic zero.
\medskip

If char $k = p \mid r$, then one does not have a finite \'etale cover as claimed in the original proof. This is attributed to the fact that the field extension $k(t) \hookrightarrow k(t)[y]/(y^r +t^s)$ fails to be separable. But in the light of coprimality of the integers $r$ and $s$, we can rework the proof as follows. Suppose $p\mid r$, we can then consider the finite \'etale cover of degree $s$ with respect to the other coordinate,
    \begin{align*}
    & C:= \Spec k[t^{\pm 1},y]/(y^r +t^s) \to\Spec k[t^{\pm 1}] := C_0 \\
    & \hspace{36mm} (y,t)\mapsto y    
    \end{align*}
where we can still extract the roots (i.e., $(y^r)^{1/s}$) since in this case $p\nmid s$ by the coprimality.

\begin{lemma}\label{exist:algspace}
For any field $k$, there exists a smooth algebraic space $\delta: \frak{S}\to \AA^2_k\bs \{0\}$ such that for every $s\geq 1$, the map $q: \KK(s)\bs L\to \AA^2_k\bs \{0\}$ factors through an \'etale-locally trivial $\AA^1$-bundle $\rho: \KK(s)\bs L\to \frak{S}$.
\end{lemma}
The strategy of the proof follows from a more general principle involved in constructing an algebraic space from a surface by replacing a desired curve by its finite \'etale covering (\cite[\S 1.0.1]{duboulozfinston2014}). Let us now explain the working strategy following the proof of \cite{DF18} adapted to our setting.
\begin{proof}
Assume that char $k= p \nmid r$. Observe that the strictly quasi-affine threefolds $\KK\bs L$ are covered by two principal affine opens, (say) $V_x=\{x\ne 0\}$ where it is a trivial $\AA^1$-bundle as already observed, and $V_t = \{t\neq 0\}$ where the suitable modifications discussed above have to be implemented. Hence, it suffices to construct an algebraic space $\mathfrak{S}_t$ locally on $V_t$ and prove that it extends compatibly to the desired algebraic space. To this end, we prove that there (locally) exists an algebraic space $\delta_t: \mathfrak{S}_t \to U_t$ such that $q|_{V_t}: V_t \to U_t$ factors through an \'etale locally trivial $\AA^1$-bundle $\rho_t: V_t \to \frak{S}_t$. Moreover, such a factorization via $\delta_t$ should be an isomorphism when restricted to the overlap $U_{xt}:= U_x\cap U_t$. Here, $U_x= \Spec k[x^{\pm1},t]$ and $U_t = \Spec k[x,t^{\pm1}]$ are corresponding principal affine open subsets in $\AA^2\bs\{0\}$. The desired algebraic space $\delta: \mathfrak{S}\to \AA^2\bs \{0\}$ is then obtained by gluing $U_x$ and $\frak{S}_t$ by the identity along $U_{xt}$ and such that $\delta_t$ restricts to an isomorphism over $U_{xt}$.
\[ \begin{tikzcd}
	{\mathcal{K}(s)\backslash L } && {\mathbb{A}^2\backslash\{0\}} \\
	\\  {\mathfrak{S}}
	\arrow["q", from=1-1, to=1-3]
	\arrow["\rho"', from=1-1, to=3-1]
	\arrow["\delta"', from=3-1, to=1-3]
         \end{tikzcd}
        \hspace{12mm}
        \begin{tikzcd}
	{V_t} && {U_t} \\
	\\	{\mathfrak{S}_t}
	\arrow["q_{\vert_{V_t}}", from=1-1, to=1-3]
	\arrow["\rho_t", swap, from=1-1, to=3-1]
	\arrow["\delta_t", swap, from=3-1, to=1-3]
\end{tikzcd} \]
Since char $k=p \nmid r$, the map $h_0:C_1 \to C_0$ with $C_1:= \Spec k[t^{\pm1}][y]/ (y^r+t^s)$ is finite \'etale. Let $h:\Spec R =: C \to C_0$ be the Galois closure of $h_0$. In other words, $C$ gives the normalization of $C_1$ in the Galois closure $\kappa$ with respect to the field extension 
           $$k(t)\hookrightarrow k(t)[y]/ (y^r+t^s). $$
The above-mentioned Galois closure $\kappa$ is precisely obtained by extracting the $r$th-roots from $t^s$, which is valid since $ p\nmid r$ by assumption. However, by the coprimality of $r$ and $s$, this is equivalent to extracting the $r$th-roots of $t$ along with all the $r$th-roots of $-1$. Thus, neither $\kappa$ nor the curves $C_1$ or $C$ depend on $s$. By construction, this gives us an \'etale $G$-torsor $h: C\to C_0$ with $G$ being the Galois group $\text{Gal}(\kappa/k(t))$ and a factorization 
        $$h:C\xrightarrow{h_1} C_1\xrightarrow{h_0} C_0$$ 
where $h_1:C\to C_1$ is an étale torsor under a subgroup $H\leq G$ with $[G:H]=r$. To conclude, it suffices to show that there exists a scheme $S$, obtained by gluing $r$ copies of $S_{\barg}$, for $\barg\in G/H$, along the curves $C_0\times C$ via the identity, such that $G$ acts freely on $S$. This action gives rise to a geometric quotient $\mathfrak{S}_t:=S/G$ in the category of algebraic spaces, with $S\to\mathfrak{S}_t$ an étale $G$-torsor. To this end, it remains to verify that the local morphisms glue to a global $G$-invariant morphism, which then descends to the corresponding global quotient. These follow verbatim from the construction provided in \cite[Lemma 3.2]{DF18} resulting in the existence of such algebraic space $\mathfrak{S}$.
\end{proof}

%---------------------------------------
\subsection*{Motivic Brouwer degree}
Let $g:\mcal{K}\bs L\to \AA^2\bs \{0\}$ be the $\AA^1$-weak equivalence from \cref{g-is-an-A1-weq}. Recall that our aim is to show that the morphism $i: \AA^2\bs \{0\} \to \KK\bs L$ is an $\AA^1$-weak equivalence in $\Spc_k$. To this end, it is enough to show that the composite
$$\Abs\xrightarrow{i} \mcal{K}\bs L \xrightarrow{g}\Abs $$
is an $\AA^1$-weak equivalence in $\Spc_k$. To show that $i\circ g$ is an $\AA^1$-weak equivalence, we use the associated $\AA^1$-Brouwer degree map from the motivic homotopy theory (cf. \cref{thm:A1-Brouwer}). The morphism $g$ being an $\AA^1$-weak equivalence, we can identify $\mcal{K}\bs L$ in the Milnor-Witt $K$-theory as $H^1(\mcal{K}\bs L, K_2^{\MW}) = K_0^{\MW}(k)\cdot \mu$, where $\mu = g^{*}(\xi)$ is the pullback of the explicit generator. In the view of \cref{A1degreemap}, it is enough to show that the induced map on the Milnor-Witt $K$-theory sheaves
$$i^*: H^1(\mcal{K}\bs L, K^{\MW}_2) \to  H^1(\Abs,K^{\MW}_2)$$ 
is an isomorphism. The \cref{setup:diagram1} induces the following commutative diagram in Milnor-Witt $K$-theory
\begin{equation}
\begin{tikzcd}[column sep= 6pt, row sep=15pt]
	{H^1(\mathcal{K}\backslash L, K_2^{\MW})} && {H^2(\mathcal{K}/(\mathcal{K}\backslash L), K_2^{\MW})} && {H^2(\mathcal{K},K_2^{\MW})} && {H^2(\mathcal{K}\backslash L, K_2^{\MW})} \\
	\\
	{H^1(\mathbb{A}^2\backslash \{0\},K_2^{\MW})} && {H^2((\mathbb{P}^1)^{\wedge 2}, K_2^{\MW})} && {H^2(\mathbb{A}^2, K_2^{\MW})} && {H^2(\mathbb{A}^2\backslash \{0\}, K_2^{\MW})}
	\arrow["\partial", from=1-1, to=1-3]
	\arrow["{i^*}"', from=1-1, to=3-1]
	\arrow[from=1-3, to=1-5]
	\arrow["{j^*}"', from=1-3, to=3-3]
	\arrow[from=1-5, to=1-7]
	\arrow[from=1-5, to=3-5]
	\arrow["{i^*}"', from=1-7, to=3-7]
	\arrow["{\partial'}"', from=3-1, to=3-3]
	\arrow[from=3-3, to=3-5]
	\arrow[from=3-5, to=3-7]
\end{tikzcd}
\end{equation}
Since $j$ is an $\AA^1$-weak equivalence, this implies that $j^*$ is an isomorphism. By $\AA^1$-homotopy invariance of Milnor-Witt $K$-theory, we have that $H^*(\AA^2, K_2^{\MW})=0$, which due to the exactness renders that $\partial'$ is also an isomorphism. Observe that the groups $H^1(\mcal{K}\bs L, K^{\MW}_2)$ and $H^2(\mcal{K}/(\mcal{K}\bs L), K^{\MW}_2)$ are actually free $K^{\MW}_0$-modules of rank 1 due to \cite[Lemma 4.5]{asok2014algebraicspheres}. This implies that $\partial$ is $K^{\MW}_0(k)$-linear and, in particular, an injective group homomorphism. Thus, it remains to only verify that $\partial$ is surjective which has been proven in \cite[Proposition 3.3]{DF18} over fields of characteristic zero, substantially using description of the generators and relations in Milnor-Witt $K$-theory. We observe that their proof can be upgraded to perfect fields as well.

%--------------------------------------
\begin{prop}\label{cocyle-is-surjective}
Let $k$ be a perfect field. Then the connecting homomorphism 
  $$\partial: H^1(\mcal{K}\bs L, K^{\MW}_2)\to H^2(\mcal{K}/(\mcal{K}\bs L),K^{\MW}_2)$$
is surjective. 
\end{prop}
\begin{proof}
Due to \cite[\S 2.2]{fasel2020Chow-Witt}, we have the relation $H^n(\KK /(\KK\bs L), K^{\MW}_*) \cong H_L^n(\KK,K^{\MW}_*)$ holds over perfect fields and for all $n\in \ZZ$, which for $n=2$ helps us to identify the generators of the group $H^2(\KK/(\KK\bs L), K^{\MW}_2)$ with that of $H_L^2(\KK,K^{\MW}_2)$. Recall from \cref{A1-weakeq-of-j} that the normal bundle of $L$ in $\KK$ is explicitly generated by the sections $y$ and $t$. Using the localization exact sequence of Chow-Witt theory (\cite[\S 10.4]{fasel2008groupeC-W}), this provides us with a long exact sequence and, as a result, also the generators for $H^2_L(\KK,K^{\MW}_2)$ explicitly as the following class of cocycle:
\begin{equation}
    \beta:= \langle1\rangle\otimes \bar{t} \otimes \bar{y} \in K^{\MW}_0(\kappa(L), \wedge^2 \mathfrak m_L/ \frak m^2_L). 
\end{equation}
To prove our claim, it is enough to show that this cocycle is indeed a boundary of a cocycle $\alpha \in H^1(\KK\bs L, K^{\MW}_2)$. The proof exploits several machinery developed in \cite{morel2012A1topology}. Consider the integral subvarieties of $\KK$ as follows: 
\begin{align*}
 & M:= \{y=0\}\subset \KK  & N:=\{t=0\}\subset \KK \quad\quad   & L':=\{y=t=0,\ x^{m-1}z=1 \}\subset \KK
\end{align*}
By the construction, we have $M\cap N = L\bigsqcup L'$. By the definition of $M$ and $N$, we have that the element $\gamma:= [y]\otimes \bar{t}\in K^{\MW}_1(\kappa(N), \frak{m}_N/\frak{m}^2_N)$ has a non-trivial boundary only on $L$ and $L'$ with same values of the boundaries in respective residue fields; the boundary of $\gamma$ on $L$ is given by $\langle1\rangle \otimes \bar{t} \wedge \bar{y}\in K^{\MW}_0(\kappa(L),\wedge^2 (\frak{m}_L/ \frak{m}^2_L))$ and on $L'$ given by $\langle1\rangle \otimes \bar{t} \wedge \bar{y}\in K^{\MW}_0(\kappa(L'),\wedge^2 (\frak{m}_{L'}/ \frak{m}^2_{L'}))$. Eventually, it turns out that $\gamma \in  K^{\MW}_1(\kappa(N),\frak{m}_N/\frak{m}_N^2)$ is not a cocycle on $\KK \bs L$, but we can produce one by doing the following modification. Note that the symbol
    $$\alpha':= [x^{m-1}z-1]\otimes \bar{y}\in K^{\MW}_1(\kappa(M), \frak{m}_M/\frak{m}^2_M)$$ 
has a non-trivial boundary only on $L'$. Indeed from the equation of $\KK$, we have that $x(x^{m-1}z-1) = y^r+t^s$ and again since $r,s$ are coprime, we have that any prime containing $x^{m-1}z-1$ and $y$ must also contain the $t$ and in turn the prime ideal corresponding to $L'$ (from the definition of $L'$). We now compute the boundary of $\alpha'$ on $L'$, taking advantage of the above-mentioned primality condition of the ideals. Since $x$ is invertible in $L'$, we can write 
    $$(x^{m-1}z -1) = (y^r+t^s)/x = x^{-1} y^r+ x^{-1} t^s \in \mcal{O}_{\frak{m}_{L'}}.$$ 
Therefore, the boundary of $\alpha'$ is the same as the boundary of $[x^{-1}t^s]\otimes \bar{y}$ on $L'$. The latter can be computed using the explicit relations given in \cite[Lemma 3.5]{morel2012A1topology}. In particular, we have that 
        $$[x^{-1}t^s] = [x^{-1}] + \langle x^{-1} \rangle [t^s] = [x^{-1}] + \langle x \rangle [t^s].$$
Now due to \cref{MW:epsilon-defn}, we have that $[t^s] = s_{\epsilon}[t]$. Finally, using the boundary formulas as stated in \cite[Proposition 3.17]{morel2012A1topology}, we obtain the boundary of $\alpha'$ as 
$$d (\alpha') =  \langle x \rangle s_{\epsilon}\otimes \bar{y}\wedge \bar{t} = \langle -x\rangle s_{\epsilon}\otimes \bar{t}\wedge \bar{y}.$$
Set $S:= \{x^{m-1} z=1 \}\subset \KK$ as a codimension 1 subvariety of $\KK$. Since $M$ and $S$ are different codimension 1 subvarieties, we compute the boundary as follows:
$$d\big([y, x^{m-1}z-1]\big) = [x^{m-1}z-1]\otimes \bar{y} +\epsilon [y]\otimes \overline{x^{m-1}z-1} $$
As $d^2=0$ and $\epsilon = -\Lin -1\Rin$, we have
    $$ d\big([y] \otimes \overline{x^{m-1}z-1}\big) = \langle -1 \rangle\ d( [x^{m-1}z-1]\otimes \bar{y}) = \langle x \rangle s_{\epsilon}\otimes \bar{t}\wedge \bar{y}.$$ 
Now $x$ is an unit in $S$, we have that 
    $$d(\Lin x\Rin [y]\otimes \overline{x^{m-1}z-1}) = s_{\epsilon} \otimes \bar{t}\wedge \bar{y}.$$ 
Similarly, 
    $$d([x^{m-1}z-1] \otimes \bar{t}) = \Lin x\Rin r_{\epsilon} \otimes \bar{t}\wedge \bar{y}$$ 
and as a consequence, we have 
    $$d([t]\otimes \overline{x^{m-1}z-1}) = \Lin -1\Rin d([x^{m-1}z-1]\otimes \bar{t})= \Lin -x \Rin r_{\epsilon}\otimes \bar{t}\otimes \bar{y}.$$ 
Thus, 
    $$d(\Lin-x\Rin[t]\otimes x^{m-1}z-1) =r_{\epsilon} \otimes \bar{t}\wedge \bar{y}.$$
As gcd $(r,s)=1$, we can choose $g,h \in \mathbb{N}$ such that $gr-hs=1$. For any integers $p,q\in \ZZ$, we have $(pq)_{\epsilon} = p_{\epsilon} q_{\epsilon}$, and so we have $g_{\epsilon} r_{\epsilon} - h_{\epsilon} s_{\epsilon} = \Lin \pm 1 \Rin$ \footnote{it is odd if $gr$ is odd and $\Lin-1\Rin$ otherwise}. Altogether, we have obtained that 
    $$d(g_{\epsilon}\Lin-x\Rin [t]  \otimes \overline{x^{m-1}z-1}) - h_{\epsilon}\Lin x\Rin [y] \otimes (\overline{x^{m-1}z-1})) = \Lin \pm1 \Rin \otimes \bar{t}\wedge \bar{y}. $$ 
This gives us the desired cocycle $\alpha$ by modifying $\alpha'$ as
    $$\alpha:=  [y]\otimes \bar{t} - \Lin\pm 1\Rin \bigg(g_{\epsilon} \Lin -x \Rin[t]\otimes \overline{(x^{m-1}z-1)} - h_{\epsilon} \Lin x\Rin[y]\otimes \overline{(x^{m-1}z-1)}\bigg)$$
that maps to the generator of $H^2_L(\KK,K^{\MW}_2)$ under $\partial$.
\end{proof}

\begin{remark}
When char $k=0$, the authors in \cite{DF18} also present another proof, which they call the \emph{lazy proof}, exploiting the stable $\AA^1$-contractibility of $\KK$ proven in \cite[Theorem 4.2]{HKO16}. In brief, the proof in the stable $\AA^1$-homotopy category heavily relies on the fact that every smooth quasi-projective scheme is dualizable in $\SH(k)$. This is, in turn, proven by exploiting the existence of resolution of singularities in characteristic zero. Hence, provided the existence of resolution of singularities over perfect fields, the lazy proof works verbatim in our context. We also remark that the above surjectivity can be lifted to arbitrary fields provided the extension of framework of \cite{morel2012A1topology}.
\end{remark}
%---------------------------------------------

\subsection{Relative $\AA^1$-contractibility of $\KK$}
All these results allow us to extend the $\AA^1$-contractibility of $\KK$ in the unstable $\AA^1$-homotopy category over generalized base schemes.

\begin{prop}\label{A1-cont-over-perfect-fields}
For a perfect field $k$, the canonical morphism of the Koras-Russell threefolds of first kind $f: \KK\to \Spec k$ is an $\AA^1$-weak equivalence in $\Spc_k$.
\end{prop}
\begin{proof}
The arguments from \cref{A1-weakeq-of-j} and \cref{A1-weakeq-of-i} show that the morphisms $i$ and $j$ are $\AA^1$-weak equivalences in $\Spc_k$. Consequently, we get that $\phi:\AA^2 \to \KK$ is an $\AA^1$-weak equivalence in \cref{setup:diagram1} and the conclusion follows from the $\AA^1$-contractibility of $\AA^2$.
\end{proof}

%-------------------------------------
\begin{theorem}\label{KR3Fmain:Noetherian}
Let $S$ be any Noetherian base scheme with perfect residue fields and let $\KK$ be the family of Koras-Russell threefolds of the first kind of relative dimension 3 over $S$. Then $f: \KK \to S$ is an $\AA^1$-weak equivalence in $\Spc_S$.
\end{theorem}

\begin{proof}
Let us first observe that the canonical map $f: \KK \to S$ is well-defined and smooth. Choose a point $s \in S$ with residue fields $\kappa_s$. Now, if $\kappa_s$ has characteristic zero, then the $\AA^1$-contractibility of $\KK \to \Spec \kappa_s$ follows from \cite[Theorem 1.1]{DF18}. If, on the other hand, char $\kappa_s = q >0$, it follows from \cref{A1-cont-over-perfect-fields} that $\KK \to \Spec \kappa_s$ is $\AA^1$-contractible since by assumption we have that $\kappa_s$ is a perfect field. In all, we have that all the fibers of the morphism $f: \KK \to S$ are $\AA^1$-contractible over corresponding residue fields, and the proof follows as a consequence of the gluing lemma (\cref{pointwise:phenomenon}).
\end{proof}

As a consequence of \cref{KR3Fmain:Noetherian}, we have the first examples of exotic $\AA^1$-contractible threefolds over $\Spec \ZZ$.

\begin{corollary}\label{KR3FoverZ}
The family of Koras-Russell three folds of the first kind $f: \KK\to \Spec \ZZ$ is relatively $\AA^1$-contractible in $\mathcal{H}(\ZZ)$. In particular, $\KK$ is an $\AA^1$-contractible scheme in $\HH(\ZZ)$ that is not isomorphic to $\AA^3_{\ZZ}$.
\end{corollary}
\begin{proof}
The fiber of $f$ over the generic point $o \in \Spec \ZZ$ is $\AA^1$-contractible, as the quotient field $\Spec \QQ$ has characteristic zero, for which the proof follows from \cite[Theorem 1.1]{DF18}. As a result, we have that the structure morphism $f_0:\KK_{o}\to \Spec \QQ$ is an $\AA^1$-weak equivalence in $\Spc_{\QQ}$. On the other hand, for all closed points $q \in \Spec\ZZ$, the corresponding residue fields being isomorphic to $\Spec \FF_q$ (which, in particular, is perfect), the $\AA^1$-contractibility of the closed fibers $f_q: \KK_q \to \Spec \FF_q$ follows from \cref{A1-cont-over-perfect-fields}.
\[\begin{tikzcd}
	& {\KK_o} & \KK & {\KK_{q}} \\
	\\
	{\Spec \QQ \cong \Spec \kappa_{o}} && {\Spec \ZZ} && {\Spec \kappa_q \cong \Spec \FF_q}
	\arrow[from=1-2, to=1-3]
	\arrow["{{f_o}}"', from=1-2, to=3-1]
	\arrow["f", from=1-3, to=3-3]
	\arrow[from=1-4, to=1-3]
	\arrow["{{f_q}}", from=1-4, to=3-5]
	\arrow[from=3-1, to=3-3]
	\arrow[from=3-5, to=3-3]
\end{tikzcd}\]
Altogether, $f:\KK\to \Spec \ZZ$ is a morphism of motivic spaces such that all its fibers are relatively $\AA^1$-contractible over the corresponding residue fields and the conclusion again follows from \cref{pointwise:phenomenon}.
\end{proof}

%---------------------------------------------
\section{Relative \texorpdfstring{$\mathbb{A}^1$}{A1}-contractibility of Koras-Russell Prototypes over a base scheme} \label{sec:Gen-KR3F}

Following the relative $\AA^1$-contractibility of Koras-Russell threefolds, we now extend the relative $\AA^1$-contractibility of the generalized Koras-Russell varieties introduced and studied by the authors in \cite{dubouloz2025algebraicfamilies}. As for the notation, $k^{[n]}$ will denote the polynomial ring in $n$-variables. Let us begin with the definition of these varieties.
\begin{defn} \label{defn:GKR3F}
Let $r,s \geq 2$ be any coprime integers. Consider the triple $(m,\ul{n},\psi)$ with integers $m\ge 0$ and a $(m+1)$-tuple $\ul{n} = (n_0,n_1,\dots,n_m)$ of multi-index with $n_i>1$ and a polynomial $\psi \in k^{[1]}$ such that $\psi(0)\neq 0$. Then consider the coordinate ring 
    $$R_m(\ul{n},\psi):= k[x_0, x_1\dots,x_m,y,z,t]/ \big(\ul{x}^{\ul{n}}z + y^r+t^s+x_0\ \psi(\ul{x}) \big)$$
where $\ul{x}= \prod_{i=0}^{m}x_i$ and $\ul{x}^{\ul{n}}= \prod_{i=0}^{m} x_i^{n_i}$. The \emph{generalized Koras-Russell varieties} are defined as the smooth affine variety associated to this coordinate ring, i.e.,
\begin{equation}\label{fulleqn:GK-R3F}
    \XX_m(\ul{n},\psi):= \Spec R_m(\ul{n},\psi)
\end{equation}        
\end{defn}
For every $m\ge 0$, the so-obtained affine variety $\XX_{m}(\ul{n},\psi)$ is of dimension $(m+3)$ and its smoothness is a consequence of the Jacobian criterion owing to the assumption that $\psi(0)\neq 0$. Let us now look at a concise example of one such variety in relative dimension $(m+3)$ constructed over the affine space $\AA^{n-3}_k$, for $n\ge 4$. Put $\psi(\ul{x}) = 1+\ul{x}+ \sum_{i=2}^{n-2}a_i \ul{x}^i$ which is clearly non-vanishing at the origin, we get that
\begin{equation}\label{deformed-example}
\XX_m(n,\psi) := \{\ul{x}^nz = y^r+t^s+x_0 \big(1+\ul{x}+\sum_{i=2}^{n-2}a_i \ul{x}^i\big)\} \subseteq \AA^{n-3}_k\times \AA^{m+3}_k
\end{equation}
with coordinates given by $\Spec(k[a_2,\dots,a_{n-2}][x_0,\dots,x_m,y,z,t])$. This makes it a smooth affine scheme under the projection map $\XX_m(n,\psi) \to \AA^{n-3}_k$ given by projecting onto the coefficients $\{a_i\}$.
% ------------------------------------------------------------------
\begin{prop}\label{KR3Fprototypes:field-A1-cont:perfect}
Let $k$ be any perfect field. Then under the notations and assumptions as in \cref{fulleqn:GK-R3F}, the canonical morphism $\XX_{m}(\ul{n},\psi) \to \Spec k$ is an $\AA^1$-weak equivalence in $\Spc_k$.
\end{prop}

The authors in \cite{dubouloz2025algebraicfamilies} have proven the $\AA^1$-contractibility of these varieties over fields of characteristic zero. We  now promote this result over perfect fields in the light of \cref{A1-cont-over-perfect-fields} following their original proof. To begin with, let us note that all these varieties are all stably $k$-isomorphic (\cite[Proposition 3]{dubouloz2025algebraicfamilies}), that is, $\XX_{m}(\ul{n},\psi)\times \AA^1_k \cong \XX_m(\ul{n},1)\times \AA^1_k$, for any polynomial $\psi\in k^{[1]}$ as above. Hence, in the light of $\AA^1$-homotopy invariance, it is enough to show that $\XX_m:= \XX_m(\ul{n},1)$ is $\AA^1$-contractible. Observe that when $m=0$ and $\psi(\ul{x})=1$, $\XX_{0}(n_0,1)$ is precisely the family of Koras-Russell threefolds studied in the previous section, which we proved to be $\AA^1$-contractible (\cref{A1-cont-over-perfect-fields}) over perfect fields. 

\begin{proof}[Proof of \cref{KR3Fprototypes:field-A1-cont:perfect}]
Let us fix 
\begin{align*}
& A_{m-1} = k[x_1,\dots,x_{m-1},y,t] & \text{and}\ &&  R_m = \frac{A_{m-1}[x_0,x_m,z]}{\prod_{i=0}^{m}x_i^{n_i}z + y^r +t^s +x_0}.
\end{align*}
Then consider the following closed subschemes of $\XX_m$
\begin{align}
    & W_m = \{x_m=0\}\cong \Spec (R_m/x_m R_m)\cong \Spec(A_{m-1}[z]) \cong \AA^{m+2}_k \label{W_m} \\
    & H_m = \{z=0\}\cong \Spec (R_m/z R_m)\cong \Spec(A_{m-1}[x_m])\cong \AA^{m+2}_k  \label{H_m}\\
    & P_m = W_m\cap H_m \cong \Spec (R_m/(x_m, z) R_m)\cong \Spec(A_{m-1}) \cong \AA_k^{m+1} \label{P_m}
\end{align}
Let $j_m: H_m\to \XX_m$ and $i_m: H_m\bs P_m\to \XX_m \bs W_m$ be the canonical closed immersions. By setting $Z:= x_m z$, we have an isomorphism 
$$\XX_m\bs W_m \cong \Spec\bigg(A_{m-2}[x_0,x_{m-1},Z][x^{-1}_m]/(x_0^{n_0}\prod_{i\neq m}x_i^{n_i}Z +y^r+t^s+x_0)\bigg) \cong \XX_{m-1}\times \GG_{m,k} $$
for which $i_m$ equals the product 
$$j_{m-1}\times \id_{\GG_{m.k}}: H_m\bs P_m \cong H_{m-1}\times_k \GG_{m.k} = \Spec\big(A_{m-2}[x_{m-1}][x_m^{\pm1}]\big)\to \XX_{m-1}\times_k \GG_{m.k}.$$
By setting up the cofiber sequence as in \cref{setup:diagram1}, we get the following diagram
\begin{equation} \hspace{-15mm}
\begin{tikzcd}
	{H_m\bs P_m} && {H_m} && {H_m/ (H_m\bs P_m)} \\ \\
	{\XX_m\bs W_m} && {\XX_m} && {\XX_m/(\XX_m\bs W_m)}
	\arrow[from=1-1, to=1-3]
	\arrow["{i_m}"', from=1-1, to=3-1]
	\arrow[from=1-3, to=1-5]
	\arrow["{j_m}", from=1-3, to=3-3]
	\arrow["{l_m}",  from=1-5, to=3-5]
	\arrow[from=3-1, to=3-3]
	\arrow[from=3-3, to=3-5]
\end{tikzcd}
\end{equation}
associated to the open immersions $H_m\bs P_m\to H_m$ and $X_m\bs W_m\to W_m$. Observe that $H_m$ is $\AA^1$-contractible by \cref{H_m} and so $\XX_m$ is $\AA^1$-contractible provided $j_m$ is an $\AA^1$-weak equivalence. For $m=0$, the morphism $j_0:H_0\to \XX_0$ is an $\AA^1$-weak equivalence due to \cref{A1-cont-over-perfect-fields}. We now establish this property for $m\ge 1$ by induction on $m$. 
\medskip

To show that $j_m$ is an $\AA^1$-weak equivalence, it suffices to show that $i_m$ and $l_m$ are $\AA^1$-weak equivalences in $\Spc_k$ in the light of \cref{weak5lemma}. The $\AA^1$-weak equivalence of $i_m$ follows from the fact that $i_m\simeq j_{m-1}\times \id_{\GG_{m,k}}$ and  by the induction hypothesis, $j_{m-1}$ is an $\AA^1$-weak equivalence. On the other hand, $l_m$ is also an $\AA^1$-weak equivalence. Indeed, as $P_m$ is obtained as the transversal intersection of two smooth subvarieties inside $\XX_m$, the normal cone $N_{P_m}H_m$ of $P_m$ in $H_m$ is in fact the restriction of the normal cone $N_{W_m}\XX_m$ of $W_m$ in $\XX_m$ to $P_m$. But the normal cone of $N_{W_m}\XX_m$ is trivial (attributed to the similar reasoning as explained in \cref{A1-weakeq-of-j}), whence the triviality of $N_{P_m}H_m$. The proof now follows from the homotopy purity \cite[Theorem 2.23]{MV99} as all these varieties are smooth. In particular, 
$$H_m/(H_m\bs P_m)\simeq \Th(N_{P_m}H_m) \simeq P_m\wedge \PP^1$$ 
and similarly, 
$$\XX_m/(\XX_m\bs W_m) \simeq \Th(N_{W_m}\XX_m) \simeq W_m\wedge \PP^1.$$
The map $l_m$ coincides with the following map observed under the above-mentioned isomorphisms
$$\Th(N_{P_m}H_m)\simeq P_m\wedge \PP^1 \xrightarrow{\iota\wedge \PP^1} W_m\wedge \PP^1\simeq \Th(N_{W_m}\XX_m).$$
obtained as the $\PP^1$-suspension of the closed immersion $\iota: P_m\hookrightarrow W_m$. Hence, the $\AA^1$-weak equivalence of $\iota$ implies that of $l_m$.
\end{proof}
% -------------------------------------------------------------
\begin{theorem}\label{KR3F:prototypes-A1-cont-base:Noetherian}
Let $S$ be any Noetherian scheme with perfect residue fields. Then under the assumptions as in \cref{fulleqn:GK-R3F}, the canonical morphism $\varphi: \XX_{m}(\ul{n},\psi) \to S$ is an $\AA^1$-weak equivalence in $\Spc_S$.
\end{theorem}
\begin{proof}
The proof is the higher dimensional analog of \cref{KR3Fmain:Noetherian}: choose any point $s\in S$ and consider the induced morphism on the corresponding fiber $\varphi_s: \XX_m(\ul{n},\psi)_s \to \Spec\kappa_s$, where $\XX_m(\ul{n},\psi)_s:= \XX_m(\ul{n},\psi) \times_{\Spec \kappa_s} \Spec\kappa_s$. 
\[\begin{tikzcd}
	{\XX_m(\ul{n},\psi)_s} && {\XX_m(\ul{n},\psi)} \\ \\
	{\Spec \kappa_s} && S
	\arrow[from=1-1, to=1-3]
	\arrow["{{\varphi_s} }"', from=1-1, to=3-1]
	\arrow["\varphi", from=1-3, to=3-3]
	\arrow[from=3-1, to=3-3]
\end{tikzcd}\]
Since by the hypothesis all such fields $\kappa_s$ are perfect, $\varphi_s$ is an $\AA^1$-weak equivalence in $\Spc_{\kappa_s}$ owing to \cref{KR3Fprototypes:field-A1-cont:perfect}. The proof now follows from \cref{pointwise:phenomenon}.
\end{proof}

Using the same line of argument as in \cref{KR3FoverZ}, one can derive the following result producing the first examples of exotic $\AA^1$-contractible higher dimensional arithmetic schemes.
\begin{corollary}\label{GKR3F-A1-Cont-overZ}
The map $\XX_m(\ul{n},\psi) \to \Spec \ZZ$ is an $\AA^1$-weak equivalence in $\mcal{H}(\ZZ)$. 
\end{corollary}

\begin{remark}\label{rem:AG-isnot-KR3F}
Results such as \cref{KR3Fmain:Noetherian} and \cref{KR3F:prototypes-A1-cont-base:Noetherian} enable us to produce exotic families of smooth affine varieties over fields of positive characteristic. For a comparison, let us now quickly recall the \emph{Asanuma-Gupta varieties} for the convenience of the readers. Let $k$ be a field of characteristic $p>0$ and let $f(z,t) := z^{p^e}+t+t^{sp}$ be the non-trivial line embedded into $\AA^2_k$ such that the integers $e,s \ge 1$ with $p^e \nmid sp$ and $sp \nmid p^e$. Then the \emph{Asanuma-Gupta varieties} are defined by the explicit polynomial equations
\begin{align*}
   & \mcal{AG}_3 := \Spec \bigg(\frac{k[x,y,z,t]}{x^my - f(z,t)} \bigg) \subset \AA^4_k \quad \text{where integers}\ m\ge 2  \\
   & \mcal{AG}_{m+2} := \Spec \bigg(\frac{k[x_1,\dots,x_m,y,z,t]}{x_1^{r_1}\dots x_m^{r_m}y-f(z,t)} \bigg) \subset \AA^{m+3}_k \quad \text{where integers}\ m,r_i \ge 2
\end{align*}
Observe that due to \cite[Theorem 5.1]{asanuma1987polynomial}, we have that the projection map $\widetilde{\pr}_x: \mcal{AG}_3 \to \AA^1_x$ onto the $x$-axis is an $\AA^2$-fiber space, which, in particular, says that all the fibers of $\widetilde{\pr}_x$ are smooth. In contrast, recall from \cref{sec:Rel-A1-Contr-KR3F} that for the projection $\pr_x: \KK \to \AA^1_x$, the fiber over origin is isomorphic to $C_{r,s}\times \AA^1_k$, which is clearly singular. Similarly, one can show that the $\widetilde{\pr}_{x_1}: \mcal{AG}_{m+2}\to \AA^1_{x_1}$ is an $\AA^{m+3}_k$-fiber space while $\pr_{x_0}:\XX_m(\ul{n},\psi)\to \AA^1_{x_0}$ is isomorphic to the singular variety $C_{r,s}\times \AA^{m+1}_k$ over $x_0 = 0$. Thus, these results have enormous potentials to enable us to produce distinct classes of counter-examples to the Zariski Cancellation Problem in all dimensions $\ge 3$ in positive characteristics apart from the existing ones.
\end{remark}

% ----------------------------------------------------------------------
\subsubsection*{\textbf{Moduli of Arbitrary Smooth Exotic Affine Varieties}} 
In \cite[Theorem 5.1]{asok2007unipotent}, the authors constructed an infinite collection of pairwise, non-isomorphic, exotic, $\AA^1$-contractible, strictly quasi-affine varieties of relative dimensions $\ge 4$. They also showed that such a situation cannot occur in dimensions 1 and 2 over fields of characteristic zero (\cite[Claims 5.7 and 5.8]{asok2007unipotent}). The missing gaps in this picture were whether there exists an analogous collection for threefolds and if there exists a collection consisting solely of \emph{affine} varieties in dimension $\ge 4$. The first gap was answered positively by the authors in \cite[Corollary 1.3]{DF18} by devising a general method to produce a moduli of arbitrary positive dimension of pairwise non-isomorphic, stably isomorphic, $\AA^1$-contractible smooth affine threefolds inspired from \cite{dubouloz2011noncancellation}. The second gap is also answered positively due to the recent work of \cite{dubouloz2025algebraicfamilies} exploiting the $\AA^1$-contractibility of the generalized Koras-Russell varieties. Our main results (\cref{KR3Fmain:Noetherian} and \cref{KR3F:prototypes-A1-cont-base:Noetherian}) adds a great significance to this overall picture by producing a moduli of arbitrary positive dimension of pairwise, non-isomorphic, stably isomorphic, $\AA^1$-contractible, smooth, affine schemes over general base schemes in \emph{relative} dimension $\ge 3$.

%-------------------------------------------------------
\section{Exotic Motivic Spheres} \label{sect:exotic-motivic-spheres}
Recall that in motivic homotopy theory, we have two types of sphere objects. The first model is given by the \emph{topological circle} via the simplex $S^1:= \Delta^1/\partial \Delta$, pointed by the image $\partial \Delta \to \Delta$. The other model is given by the geometric \emph{Tate circle} $\GG_m:= \AA^1 \backslash \{0\}$ which is, in fact, a smooth scheme pointed by 1. As the unstable $\AA^1$-homotopy category is symmetric monoidal with respect to the smash product, one forms the \emph{mixed motivic spheres} by smashing copies of these circles, i.e.,
    $$\SS^{p,q}:= (S^1)^{\wedge (p-q)} \wedge \GG_m^{\wedge q} \qquad \text{for all} \quad p\ge q\ge 0.$$
For notational convenience, we will adapt to write $S^{p-q}$ for $(S^1)^{\wedge (p-q)}$ and $\GG_m^q$ for $(\GG_m)^{\wedge q}$.

\begin{example}
The following are some basic examples of motivic spheres.
\begin{enumerate}
\item By definition, we have $\SS^{1,0}=S^1$ and $\SS^{1,1} = \GG_m$. More generally, for every $n\geq 1$, we have 
                $$\SS^{2n-1,n} \simeq \AA^n\bs\{0\}.$$ 
\item For every $n\geq 1$, we have that 
     $$\SS^{2n,n} =\Sigma_{S^1} \SS^{2n-1,n} \simeq \AA^n/({\AA^n\bs \{0\}}) \simeq \PP^n/(\PP^n\bs \{0\}) \simeq \PP^n/\PP^{n-1}.$$
In particular, when $n=1$, we see that the motivic sphere $\SS^{2,1} \simeq \GG_m\wedge S^1 \simeq \PP^1$ is represented by the projective line,
\item The authors in \cite{ADF2017smooth} study a certain family of smooth affine quadrics given by
\begin{align*}
& Q_{2n}:= \Spec \bigg(k[x_1,\dots,x_m,y_1,\dots,y_m] / \langle \sum_i  x_iy_i-1 \rangle \bigg)\\
& Q_{2n-1} := \Spec \bigg(k[x_1,\dots,x_m,y_1,\dots,y_m,z] / \langle \sum_i  x_iy_i-z(1+z) \rangle \bigg)
\end{align*}
and show that these quadrics in turn form smooth models for the motivic spheres
$$Q_{2n-1}\simeq \SS^{2n-1,n}\quad  \text{and}\quad  Q_{2n}\simeq \SS^{2n,n} \simeq (\PP^1)^{\wedge n}. $$
Coupling with previous examples, one sees that $\Sigma_{S^1} \AA^n\bs \{0\} \simeq  (\PP^1)^{\wedge n}$, for all $n \ge 1$,

\item For all integers $p>2q$, the motivic spheres $\SS^{p,q}$ cannot be represented by a smooth scheme \cite[Proposition 2.3.1]{ADF2017smooth}.
\end{enumerate}
\end{example}

The main objective of this section is to prove the existence of exotic motivic spheres in higher dimensions. This study, in contrast, can be seen as a "compact" analog to that of exotic affine varieties as discussed in \cref{sec:Rel-A1-Contr-KR3F} and \cref{sec:Gen-KR3F}. The inception of this study is the following question:
\begin{question}\label{existence:exotic-spheres}
Let $S$ be an arbitrary base scheme. Then for a smooth scheme $X\to S$, does $X \simeq_{\AA^1} \AA^n_S \backslash \{0\}$ imply the isomorphism $X\cong  \AA^n_S \backslash \{0\}$ of $S$-schemes?
\end{question}

In pursuit of this question, let us introduce the following definition:
\begin{defn}\label{exoticspheres:defn}
Let $S$ be a base scheme. A smooth scheme $\mcal{E} \to S$ of finite type and of relative dimension $n\ge 0$ is called an \emph{exotic motivic sphere} if it has the $\AA^1$-homotopy type as $\SS^{2n-1,n}$ but not isomorphic to $\AA^n_S \bs\{0\}$ as $S$-schemes.
\end{defn}

% --------------------------------------------------------------------
\subsection{Overview of \cref{existence:exotic-spheres}} \label{sec:overview-exotic-motsph}
In dimension 1, $\SS^{0,1}:= \GG_m$ is the unique motivic sphere up to isomorphism over any reduced base scheme. 
\begin{prop}\label{exotic-reducedbase-Gm:prop}
Over a reduced scheme $S$, if a smooth scheme $\mcal{E}\to S$ is $\AA^1$-homotopic to $\GG_{m,S}$, then it is isomorphic to $\GG_{m,S}$.
\end{prop}
\begin{proof}
Let us begin with the case of fields \footnote{The author thanks Biman Roy for his generous explanation of the proof over fields.}. Let $\mcal{E}$ be a smooth curve that is $\AA^1$-homotopic to $\GG_m$. Then we have that $\pi_0^{\AA^1}(\mcal{E})\simeq \pi_0^{\AA^1}(\GG_m) \simeq \GG_m$ which leaves us to show that $\mcal{E}$ is $\AA^1$-rigid. Suppose the contrary, then recall that we have a canonical surjection of Nisnevich sheaves $\mcal{E}\to \pi_0^{\AA^1}(\mcal{E})$ (\cite[\S 2, Corollary 3.22]{MV99}) which gives us a dominant morphism $\varphi : \mcal{E}\to \GG_m$. Since by supposition $\mcal{E}$ is not $\AA^1$-rigid, then there is a finite separable extension $L/k$ and an $\AA^1$-homotopy $H: \AA^1_L \to \mcal{E}$ such that the induced maps on the 0-section and 1-section of $\Spec L\to \mcal{E}$ do not agree $H(0) \ne H(1)$ (cf. \cite[Lemma 2.1.11]{asokmorel2011}). Hence, $H$ is a dominant map which makes the composition $\varphi\circ H: \AA^1_L\to \mcal{E}\to  \GG_m$ constant (since $\GG_m$ is $\AA^1$-rigid). In particular, $\varphi$ is constant, which contradicts the fact that $\varphi$ is dominant. For any reduced scheme $S$, it suffices to check locally on affine charts. But since for any reduced ring $R$, $\GG_{m,R}$ remains $\AA^1$-rigid, the proof follows.
\end{proof}

For a field $k$ of characteristic zero, it has been recently shown by \cite[Theorem 3.1]{Choudhury2024A1type} that in dimension 2 if an open subscheme $U$ of an affine surface $X$ is $\AA^1$-weakly equivalent to $\AA^2 \bs \{0\}$, then $U \cong \AA^2\bs\{0\}$ as $k$-schemes. This question is answered negatively in dimension 3 via the Koras-Russell threefolds of the first kind (cf. \cite[\S 3]{Choudhury2024A1type}). Thus, the question remains widely open in dimensions $\ge 4$ over fields and consequently, over a base scheme. In this paper, we close this question in all dimensions $\ge 4$ over fields using the Koras-Russell prototypes.

%-----------------------------------------------------
\subsection{Existence of exotic motivic spheres in higher dimensions} \label{subsec:exist-exotic}
We show that the prototypes studied in \cref{sec:Gen-KR3F} provide counterexamples to the \cref{existence:exotic-spheres} in every dimension $\ge 4$. Due to the previously mentioned fact that each of these varieties $\XX_{m}(\ul{n},\psi)$ is stably isomorphic, it is enough to prove this fact for the varieties $\XX_m:= \XX_m(n,1)$, noting, nevertheless, that any variety in this family indeed provides a counter-example. For simplicity, let us consider the family of smooth affine varieties defined as in \cref{deformed-example}
\begin{equation}\label{counter-eg-GKR3F}
\XX_m := \{\ul{x}^n z = y^r+t^s + x_0\} \subset \AA^{m+4}_k \cong  \Spec(k[x_0,\dots,x_m,y,z,t]).
\end{equation}
The canonical morphism $\XX_m\to \Spec k$ makes it into a smooth affine scheme of dimension $(m+3)$ for all $m\ge 0$ and $n\ge 2$ and $\ul{x}:= \prod_{i=0}^{m} x_i$. The variety $\XX_m$ is obtained by setting $\psi=1$ and the index $\ul{n}=n$ in \cref{fulleqn:GK-R3F}. Observe that $p=(1,1,\dots,1,0,1,0)\in \XX_m$ is a rational point. Define the map 
\begin{align*}
    & \hspace{18mm} \phi: \XX_m \to \AA^{m+3}_k\\
    & (x_0,\dots,x_m,y,z,t)\mapsto (x_0,\dots,x_m,y,t)
\end{align*}
(forgetting the coordinate "$z$"). Then the image $\phi(p) = (1,\dots,1,0,0)=: q \in \AA^{m+3}_k$.
\medskip

We now prove several lemmas that will useful towards our main theorem (\cref{exotic-motspheres-countereg}).

\begin{lemma}\label{tgtspaces-isomorphic}
Let $k$ be a field. For all $m\ge 0$, the induced map $d\phi_p: T_p \XX_m \to T_q\AA^{m+3}_k$ is an isomorphism of tangent spaces. 
\end{lemma}
\begin{proof}
Let $f:= f(x_0,\dots,x_m,y,z,t) = \ul{x}^n z - y^r -t^s- x_0$. We now compute the partial derivatives of $f$
\begin{align*}
\frac{\partial f}{\partial x_0}= nx_0^{n-1}x_1^n\dots x_m^n z-1 \ &&
\frac{\partial f}{\partial x_j}= nx_0^n x_1^n\dots x_j^{n-1}\dots x_m^n z, \quad \text{for all}\ 0<j\le m,
\end{align*}

\begin{align*} 
\frac{\partial f}{\partial y}= -ry^{r-1}  &&
\frac{\partial f}{\partial z}= \ul{x}^n  &&
\frac{\partial f}{\partial t}=  -s t^{s-1} 
\end{align*}
The total derivative $\nabla f$ of $f$ can be computed as
$$\nabla f = \bigg( (nx_0^{n-1}x_1^n\dots x_m^n z-1),\dots, (nx_0^n x_1^n\dots x_j^{n-1}\dots x_m^n z),\dots, -ry^{r-1}, \ul{x}^n, -st^{s-1} \bigg)$$
which at the point $p$ is given by $\nabla f|_{p} = (n-1, n,\dots,n,0,1,0)$. Recall that the tangent space of $\XX_m$ at $p$ is the kernel of the total derivative at $p$. We find that the tangent space at $p$ is given as follows: 
$$T_p \XX_m = \{(a_0,a_1 \dots,a_m,b, -(n-1)a_0 -n \overunderset{m}{{i=1}}{\Sigma}a_i,d) \mid a_i,b,d \in k \}\subset \AA^{m+4}_k$$
The induced map on the differentials $d\phi_p$ is given by the 
    $$ (a_0,a_1 \dots,a_m,b, -(n-1)a_0 -n \overunderset{m}{{i=1}}{\Sigma}a_i, d) \mapsto (a_0,a_1 \dots,a_m,b,d) $$ 
which is clearly an isomorphism. In particular, this implies that the normal bundle of $\XX_m$ at $p$ is isomorphic to that of $N_{q}\AA^{m+3}_k$, which is trivial.
\end{proof}

\begin{lemma} \label{lemma:poly-killing}
Let $L/K$ be a field extension, let $a,b,c\in L$, let $N\geq 1$. Then there exist polynomials $Y(v), T(v)$ in $L[v]$ such that 
$$ Y(1)=b \quad T(1)= c,$$
and 
\begin{equation}\label{lemma-mod-eqn}
Y(v)^r + T(v)^s +\alpha v \equiv 0 \quad (mod\ v^N).	
\end{equation}
\end{lemma}

\begin{proof}
Choose $\beta, \gamma \in L^{\times}$ such that $\beta^r + \gamma ^s =0$. Since gcd $(r,s)=1$, the characteristic of $L$ cannot divide both $r$ and $s$. In particular, either $r\beta^{r-1}$ or $s \gamma^{s-1}$ is non-zero in $L$. Set
$$Y_0(v)= \beta+(b-\beta)v \qquad T_0(v)= \gamma+(c-\gamma)v. $$ 
Let us observe that the constant term of $Y_0^r + T_0^s + \alpha v$ is zero. Without loss of generality, suppose that $\lambda:= r \beta^{r-1} \ne 0$ and assume that the coefficients of $v^i$ has been killed for $i= 0,1,\dots, j-1$. If $\omega_j$ is the coefficient of $v^j$ in the current expression, we replace $Y$ by 
$$ Y(v)- \frac{\omega_j}{\lambda} v^j(1-v).$$
The added term is divisible by $v^j$, so it does not alter lower-order coefficients. In other words, modulo $v^{j+1}$, its contrubution to $Y^r$ is exactly $-\omega_jv^j$. Since it also vanishes at $v=0$ and $v=1$, the end points are preserved. By iterating this for $j=1,\dots N-1$, we have \cref{lemma-mod-eqn}. Suppose if $s\gamma^{s-1}\ne 0$, we proceed by replacing $T$ by $Y$ and proceed in the same fashion.
\end{proof}

% ------------------------------------
% We will show that any point on the fiber over $x_0=0$ can be connected to the point $o= (0,1,\dots, 0,0,0)$ and that any other point from other fibers can be connected to the fiber over $x_0=0$.
\begin{lemma}\label{X-p-isA1-connected}
For any field $k$, the smooth quasi-affine variety $\XX_m\bs \{p\}$ is $\AA^1$-chain connected over $k$.
\end{lemma}
\begin{proof}
We first prove that the smooth affine variety $\XX_m$ is 
$\AA^1$-chain connected and derive the same for the punctured quasi-affine variety $\widetilde{\XX_m}:= \XX_m\bs \{p\}$. Let $L/k$ be any arbitrary field extension and fix $\xi= (\alpha, a_1,\dots,a_m,b,d,c)\in \XX_m(L)$ such that it satisfies 
\begin{equation}\label{lemma:apply-point-q}
	(\alpha a_1\cdots a_m)^nd = b^r+c^s+\alpha.	
\end{equation}
We will now connect every such $\xi$ to the fixed point $o= (0,1,\dots, 1,0,0,0)$.
\smallskip

\noindent\textbf{Case 1:} $\alpha =0$. Then $b^r+t^s=0$. The morphism 
$$ u \mapsto (0,a_1,\dots,a_m,u^sb,d,u^rc)$$
lies in $\XX_m$, since $ (u^sb)^r + (u^rc)^s = u^{rs}(b^r+c^s)=0$. It joins $\xi$ with $y=t=0$. With $x_0=y=t=0$, the \cref{lemma:apply-point-q} imposes no further conditions on $x_1,\dots,x_m,z$ and so these coordinates can be varied linearly to obtain $o$.
\smallskip

\noindent\textbf{Case 2:} $\alpha \neq 0$ and some $a_i=0$. The \cref{lemma:apply-point-q} becomes $b^r+c^s+\alpha=0$. Because $x_0x_1\cdots x_m=0$, we may first vary $z=d$ linearly to 0. Keeping $z=0$, we can then vary $x_1,\dots,x_m$ linearly to 1. Hence, it boils down to treating the case in which all $a_i\neq 0$.
\smallskip

\noindent\textbf{Case 3:} $\alpha\neq 0$ and $a_1\cdots a_m \neq 0$. Put $A=a_1\cdots a_m\in L^{\times}$. Applying \cref{lemma:poly-killing} with $N=n$, we can find polynomials $Y(v),T(v)\in L[v]$ such that
$$ Y(1)=b, \quad T(1)=c, \quad v^n \mid Y(v)^r+T(v)^s+\alpha v.$$
Define 
$$ Z(v) = \frac{Y(v)^r+T(v)^s+\alpha v}{\alpha^n A^n v^n} \in L[v].$$
Then
$$\Gamma(v) = \bigg(\alpha v, a_1,\dots,a_m, Y(v), Z(v), T(v)\bigg)$$
defines a morphism $\AA^1_L \to \XX_m$, as 
$(\alpha a_1\cdots a_m)^n Z(v) = Y(v)^r+T(v)^s+\alpha v$. At $v=1$, \cref{lemma:apply-point-q} gives $Z(1)=d$, hence $\Gamma(1)=\xi$. At $v=0$, we have $x_0=0$. Therefore, every point of $\XX_m(L)$ is connected by an $\AA^1_L$-chain to the fiber over $x_0=0$, and \textbf{Case 1} connects that fiber to $o$. Hence, $\XX_m$ is $\AA^1$-chain connected by \cref{defn:A1-chainconn}. Now, to prove that $\widetilde{\XX_m}$ is $\AA^1$-chain connected, it is enough to show every such curve connecting $\xi$ to $o$ avoids the point $p$.
\smallskip

\noindent\textbf{Subcase 1:} Let $\alpha=0$. The same proof as in \textbf{Case 1} works as $p$ has $\alpha=1$. Hence, all these curve lie in $\widetilde{\XX_m}$.
\smallskip

\noindent\textbf{Subcase 2:} Let $\alpha\neq 0$ and some $a_i=0$. Then we have $b^r+c^s+\alpha=0$. First, moving $z=d$ linearly to 0, this gives us a curve with some $x_i=0$, so it cannot meet $p$. Next, keeping $z=0$, we can move $x_1,\dots,x_m$ linearly to 1. This curve also avoid $p$ because $p$ has $z=1$. We reduced to a point $\xi'=(\alpha, 1,\dots,1,b,0,c)\in \widetilde{\XX_m}(L)$.
\smallskip

\noindent\textbf{Subcase 3:} all $a_i\neq 0$, but $a_1\cdots a_m \neq (1,\dots, 1)$. Set $A= a_1\cdots a_m$. Using the same bridge as in the \textbf{Case 3} above, we have:
$$ \Gamma(v) = \bigg(\alpha v, a_1,\dots,a_m, Y(v), Z(v), T(v)\bigg)$$
where we have applied \cref{lemma:poly-killing} with $N=n$. Since at least one constant coordinate $a_i$ is not 1, the image of $\Gamma$ cannot contain $p$. Hence, this bridge lies entirely in $\widetilde{\XX_m}$ and ends in the fiber $x_0=0$. 
\smallskip

\noindent\textbf{Subcase 4:} Let $a_1 = \cdots = a_m =1$. This is the only case that requires a significant modification. Apply \cref{lemma:poly-killing} with $N=2n$ and choose $Y,T \in L[v]$ satisfying
$$ Y(1)=b, \quad T(1)=c, \quad v^{2n} \mid Y(v)^r+T(v)^s+\alpha v.$$
Set 
$$ Z(v) = \frac{Y(v)^r+ T(v)^s+ \alpha v}{\alpha^n v^{2n}} \in L[v] $$
and define 
\begin{equation} \label{lemma:Gamma-map}
\Gamma(v) = \bigg(\alpha v, v,\dots,1, Y(v), Z(v), T(v) \bigg).
\end{equation}

\noindent Here
$$(x_0\cdots x_m)^n = (\alpha v^2)^n = \alpha^n v^{2n},$$
so \cref{lemma:Gamma-map} is a morphism $\AA^1_L \to \XX_m$. At $v=1$, \cref{lemma:apply-point-q} gives $Z(1)=d$, hence $\Gamma(1) =\xi$. At $v=0$, the point lies in the fiber over $x_0=0$. Observe that $\Gamma$ avoids $p$. Indeed, if $\Gamma(v)=p$, then the $x_1$ coordinate of \cref{lemma:Gamma-map} forces $v=1$, but then the $\Gamma(v) = \Gamma(1) = \xi \neq p$, because $\xi\in \widetilde{\XX_m}(L)$. Thus, $\Gamma(\AA^1_L) \subset \widetilde{\XX_m}$. The same construction applies to the point $\xi'$ arising in \textbf{Case 2} as well. Thus, we have shown that every point of $\widetilde{\XX_m}(L)$ is connected inside $\widetilde{\XX_m}$ to the fiber over $x_0=0$ and that the fiber is disjoint from $p$ and, by \textbf{Case 1}, is contained in a single $\AA^1_L$-chain component containing $o$. Hence, every two $L$-points of $\widetilde{\XX_m}$ can be connected by a finite $\AA^1_L$-chain, which proves the desired claim. 
\end{proof}

% ----------------------------------------------
We now have all the tools to prove our key result that answers \cref{existence:exotic-spheres}.

\begin{theorem}\label{exotic-motspheres-countereg}
Let $k$ be an infinite perfect field. Then for every $m\ge 0$, the strictly quasi-affine variety $\XX_m\bs \{p\}$ of dimension $(m+3)$ is $\AA^1$-homotopic to $\AA^{m+3}_k \bs\{q\}$ but is not isomorphic to $\AA^{m+3}_k \bs \{q\}$ as a $k$-scheme.
\end{theorem}
\begin{proof}
Let $\XX_m$ be the smooth affine variety as described in \cref{counter-eg-GKR3F} with the rational point $p = (1,\dots,1,0,1,0)\in \XX_m$ and with its image $q = (1,\dots,1,0,0)\in \AA^{m+3}_k$. We first show that $\XX_m \bs \{p\}$ is $\AA^1$-homotopic to $\AA^{m+3}\bs\{q\}$. For this, let us set up the following commutative diagram whose rows are cofiber sequences
\begin{equation}\label{setup:diagram}\hspace{-30mm}
\begin{tikzcd}
&&&&& \XX_m\bs \{p\} & {} & {\XX_m} && {\XX_m/(\XX_m\bs \{p\})} \\ \\ {} 
&&&&& {\AA^{m+3}_k \bs \{q\}} && {\AA^{m+3}_k} && {\AA^{m+3}_k/ (\AA^{m+3}_k\bs \{q\})}
\arrow[from=1-6, to=1-8]
\arrow["\phi",  from=1-6, to=3-6]
\arrow[from=1-8, to=1-10]
\arrow["\psi", from=1-8, to=3-8]
\arrow["\rho", from=1-10, to=3-10]
\arrow[from=3-6, to=3-8]	
\arrow[from=3-8, to=3-10]
\end{tikzcd}
\end{equation}
Due to \cref{KR3Fprototypes:field-A1-cont:perfect}, we have that $\psi: \XX_m \to \AA^{m+3}_k$ is an $\AA^1$-weak equivalence in $\Spc_k$. Next, we observe that $\rho$ is also an $\AA^1$-weak equivalence. Indeed, since the normal bundle associated to the closed immersion $\{q\}\hookrightarrow \AA^{m+3}_k$ is trivial, by the $\AA^1$-homotopy purity (\cite[Theorem 2.23]{MV99}), we have that                
    $$\AA^{m+3}_k/ (\AA^{m+3}_k\bs \{q\})\simeq \Spec k\ \wedge (\PP^1)^{\wedge ({m+3})} \simeq (\PP^1)^{\wedge ({m+3})}.$$ 
Since $\{p\} \in \XX_m$ and $\XX_m$ are smooth subvarieties,  due to the aforementioned purity, we have 
    $$\XX_m/(\XX_m\bs \{p\}) \simeq \Th(N_{p}\XX_m). $$ 
By \cref{tgtspaces-isomorphic}, we have that the normal bundle $N_p \XX_m$ is isomorphic to a trivial bundle and hence that 
    $$\Th(N_{p} \XX_m)\simeq \{p\}\wedge (\PP^1)^{\wedge {(m+3)}} \simeq (\PP^1)^{\wedge {(m+3)}}.$$
Hence, we have that 
$$ (\PP^1)^{\wedge {(m+3)}} \simeq \XX_m/(\XX_m \bs \{p\})\xrightarrow{\rho} \AA_k^{m+3}/(\AA_k^{m+3}\bs \{q\}) \simeq (\PP^1)^{\wedge ({m+3})}$$
which due to the fact that the map $\rho$ is induced from the map $d\phi_p$ becomes an $\AA^1$-weak equivalence in the light of \cite[Lemma 2.1]{voevodsky2003Z/2}. The $\AA^1$-weak equivalence of $\rho$ implies, by definition, that the simplicial suspension $\Sigma_{S^1} \phi: \Sigma_{S^1} \XX_m \bs\{p\} \to \Sigma_{S^1} \AA^{m+3}_k\bs\{q\}$ is an $\AA^1$-weak equivalence of motivic spaces. Now observe that $\XX_m \bs \{p\}$ is $\AA^1$-chain connected from \cref{X-p-isA1-connected} and whence by \cite[Proposition 2.2.7]{asokmorel2011} is thereby $\AA^1$-connected as well. Moreover, it is of codimension $d=m+3\ge 2$ in $\XX_m$. Thus, by \cite[Theorem 4.1]{asok2009A1-excision}, we have that the $\AA^1$-homotopy sheaves of the open subschemes $\XX_m \bs \{p\} \hookrightarrow \XX_m$ agree in a range
$$\pi_i^{\AA^1}(\XX_m\bs\{p\}) \xrightarrow{\sim} \pi_i^{\AA^1}(\XX_m) \quad \text{for all}\quad  0\le i \le d-2 $$
and in particular, $\pi^{\AA^1}_1(\XX_m\bs \{p\})$ is trivial. It follows that the simplicial suspension 
	$$\Sigma_{S^1}\phi: \Sigma_{S^1} \XX_m\bs \{p\} \to \Sigma_{S^1} \AA^{m+3}_k\bs\{q\}$$ 
is an $\AA^1$-weak equivalence, which due to the $\AA^1$-suspension isomorphism implies that $\phi$ induces isomorphism on all of its $\AA^1$-homology sheaves, whence an $\AA^1$-homology equivalence by definition. Now since both $\XX_m\bs \{p\}$ and $\AA^{m+3}_k\bs \{q\}$ are $\AA^1$- simply connected spaces, the $\AA^1$-Whitehead theorem (see e.g., \cite[Theorem 1.1]{shimizu2022relative}) implies that $\phi$ is, in fact, an $\AA^1$-weak equivalence of motivic spaces. Now, it remains to show that these quasi-affine varieties cannot be isomorphic as $k$-schemes. Suppose, on the contrary, there is an isomorphism $\sigma: \XX_m \bs\{p\}\to \AA_k^{m+3}\bs\{q\}$ with an inverse $\tau:\AA_k^{m+3}\bs\{q\}\to \XX_m\bs\{p\}$. As the point $\{p\}$ (resp. $\{q\}$) is of codimesion at least 2 in $\XX_m$ (resp. in $\AA^{m+3}_k$), the morphism $\sigma$ and $\tau$ can be continuously extended to a smooth morphism ${\sigma'}:\XX_m \to \AA_k^{m+3}$ and $\tau':\AA^{m+3}_k\to \XX_m$. Both the maps $\sigma'$ and $\tau'$ take the same values with the identity maps away from the complement of a rational point, which implies that both $\sigma'$ and $\tau'$ are isomorphisms. But due to the fact that $\ML(\XX_m)\ncong \ML(\AA^{m+3}_k)$ (cf. \cite[Proposition 3.4]{ghosh2023triviality}), this is absurd. This proves our desired claim.
\end{proof}

\begin{remark}
A recent $\AA^1$-excision result of \cite[Corollary 3.6]{shimizu2022relative} allows us to promote the above result to perfect fields provided we know the $\AA^1$-connectedness of the quasi-affine variety $\XX_m\bs \{p\}$ over finite fields as well.
\end{remark}

% -------------------------------------------------------------
\subsection*{A note on base change for motivic spheres}
To this end, we shall now briefly comment on the base change properties of motivic spheres. In the light of the existence of exotic motivic spheres over fields (\cref{exotic-motspheres-countereg}), it is natural to systematically ask for the verity over a base scheme $S$. The question of interest can be framed as follows:

\begin{question}\label{qstn:exot-mot-base-scheme}
Let $S$ be a base scheme. Then for every $m\ge 0$, are the smooth quasi-affine varieties $\XX_m\bs \{p\}$ and $\AA^{m+3}_S\bs\{0\}$ necessarily $\AA^1$-homotopic as motivic spaces? If so, are they isomorphic as $S$-schemes?
\end{question}

We anticipate that these should \emph{not} be isomorphic as $S$-schemes owing to the failure of \emph{relative} Makar-Limanov type invariants as in the case of fields. On the other hand, to show that they are indeed $\AA^1$-homotopic, one might hope for the following approaches.

\subsubsection{Fibrewise gluing approach}
One may choose to work fiberwise as in the case of $\AA^1$-contractibility. For concreteness, fix $S$ to be a base scheme with infinite perfect residue fields. Let us now decipher the obstruction to taking up this approach. For every point $s\in S$, let us consider the base change $f_s: (\XX_m \bs \{p\})_s \to \Spec \kappa_s$, where $\kappa_s$ is the corresponding residue field at $s\in S$. Thus, by \cref{exotic-motspheres-countereg}, we have that for every $s\in S$, the quasi-affine schemes $(\XX_m \bs \{p\})_s$ are $\AA^1$-homotopic to $\AA^{m+3}_{\kappa_s}\bs \{q\}$ in $\Spc_{\kappa_s}$. 
\[\begin{tikzcd}
	{\AA_{\kappa_s}^{m+3}\bs \{0\} \simeq (\XX_m\bs \{p\})_s} && {\XX_m\bs \{p\}} \\
	{\Spec \kappa_s} && S
	\arrow[from=1-1, to=1-3]
	\arrow["{f_s}"', from=1-1, to=2-1]
	\arrow["f", from=1-3, to=2-3]
	\arrow[from=2-1, to=2-3]
\end{tikzcd}\]
But since in general the $\AA^1$-homotopy classes of motivic spheres are notorious to decipher owing to the non-triviality the motivic Brouwer degree (cf. \cite[Corollary 6.43]{morel2012A1topology}), knowing that the fibers $(\XX_m\bs \{p\})_s$ are $\AA^1$-homotopic to $\AA^{m+3}_{\kappa_s}\bs \{q\}$ does not, apriori, allow us conclude that the original space $\XX_m\bs \{p\}$ is $\AA^1$-homotopic to $\AA^{m+3}_S\bs \{q\}$. This can be explained by the lack of an analogous \emph{gluing lemma} for motivic spheres in comparison to that of contractible objects as provided by \cref{pointwise:phenomenon}. In the light of such a gluing lemma for motivic spheres, it is straightforward to verify that there are exotic motivic spheres in all relative dimensions $\ge 3$ over a base schemes as well.

% -----------------------------------------------------------
\subsubsection{The relative approach}
An alternative approach would be to directly prove that the varieties in \cref{qstn:exot-mot-base-scheme} are $\AA^1$-homotopic over $S$ itself. However, this would require extending the various tools exploited in the proof of \cref{exotic-motspheres-countereg} over a base scheme, which is, to the author's knowledge, yet to be fully understood. An immediate consequence of such a formulation would be the existence of exotic motivic spheres in all relative dimensions $\ge 3$ over general base schemes.
 
\subsubsection{Uniqueness of motivic spheres in relative dimension 2}
The authors in \cite[Theorem 3.1]{Choudhury2024A1type} prove that there are no exotic motivic spheres over a field of characteristic zero. If $S$ is any Noetherian scheme with characteristic zero residue fields, then by either of those approaches outlined above, it is straightforward to prove the inexistence of exotic motivic spheres in relative dimension 2 as well. To this end, we conjecture the following:
\begin{conjecture}
Let $S$ be any Noetherian scheme with characteristic zero residue fields. Let $f: X\to S$ be a smooth scheme of relative dimension 2. Then if $X$ is $\AA^1$-homotopic to  $\AA^2_S\bs \{0\}$, then $X$ is isomorphic to $\AA^2_S\bs \{0\}$ as a $S$-scheme.
\end{conjecture}

\printbibliography[title={Bibliography}]

\end{document}